%% file: manuscript.tex
\documentclass{article}

\usepackage[english]{babel}
\usepackage[utf8]{inputenc}
\usepackage{bbm}
\usepackage{enumitem}
\usepackage{amsmath} 
\usepackage{amssymb}
\usepackage{comment}
\usepackage{algorithm}
\usepackage{algorithmic}
\usepackage{amsthm}
\usepackage{todonotes}
\usepackage{placeins}
\usepackage[compact]{titlesec}
\usepackage{graphicx}
\usepackage{xcolor}
\usepackage{transparent}
\usepackage{leftidx}
\usepackage{framed}
\usepackage{mathtools}
\usepackage{url}
\usepackage[a4paper, total={6in, 8in}]{geometry}
\usepackage{pgfplots}
\pgfplotsset{compat=1.18}
\usepgfplotslibrary{colormaps}
\usepgfplotslibrary{patchplots}
\usepgfplotslibrary{fillbetween}
\usepgfplotslibrary{external}
\usepgfplotslibrary{groupplots}

\titleformat{\section}[hang]
{\filcenter\bfseries}
{\thesection. }{0pt}{}
\titleformat{\subsection}
{\bfseries}
{\thesubsection. }{0pt}{}

\newtheorem{proposition}{Proposition}[section]
\newtheorem{remark}[proposition]{Remark}
\newtheorem{example}[proposition]{Example}
\newtheorem{definition}[proposition]{Definition}
\newtheorem{lemma}[proposition]{Lemma}
\newtheorem{corollary}[proposition]{Corollary}
\newtheorem{theorem}[proposition]{Theorem}

\newtheorem{innersummary}{Summary}



\definecolor{TolVibrantBlue}{HTML}{0077BB}
\definecolor{TolVibrantCyan}{HTML}{33BBEE}
\definecolor{TolVibrantTeal}{HTML}{009988}
\definecolor{TolVibrantOrange}{HTML}{EE7733}
\definecolor{TolVibrantRed}{HTML}{CC3311}
\definecolor{TolVibrantMagenta}{HTML}{EE3377}
\definecolor{TolVibrantGray}{HTML}{BBBBBB}

\definecolor{TolMutedBlue}{HTML}{332288}
\definecolor{TolMutedCyan}{HTML}{88CCEE}
\definecolor{TolMutedTeal}{HTML}{44AA99}
\definecolor{TolMutedGreen}{HTML}{117733}
\definecolor{TolMutedOlive}{HTML}{999933}
\definecolor{TolMutedSand}{HTML}{DDCC77}
\definecolor{TolMutedRose}{HTML}{CC6677}
\definecolor{TolMutedWine}{HTML}{882255}
\definecolor{TolMutedPurple}{HTML}{AA4499}

\pgfmathsetmacro{\xmin}{0.12}\pgfmathsetmacro{\xmax}{0.55}
\pgfmathsetmacro{\ymin}{0.30}\pgfmathsetmacro{\ymax}{0.60}
\pgfmathsetmacro{\zmin}{-0.05}\pgfmathsetmacro{\zmax}{0.35}

\input{macros.tex}

\title{A Lagrange-Newton method for solving optimal control problems on manifolds}

\author{Laura Weigl \& Anton Schiela}

\begin{document}
\maketitle
\begin{abstract}
In PDE-constrained optimization a control is applied within a partial differential equation (PDE), e.g., as a force field, to achieve a desired configuration of the PDE solution. The objective is usually to find an optimal balance between control cost and deviation of the controlled state from the desired configuration.
We consider this problem class in a setting, where the PDE is given by a geometric variational equation and the state is a mapping into a nonlinear manifold. We define a Lagrangian function corresponding to this geometric optimal control problem and derive first-order optimality conditions. Moreover, we present a Lagrange-Newton method for their numerical solution. We introduce an implementable algorithm using a local splitting technique and show its equivalence to Newton's method applied to the derivative of the Lagrangian. The theoretical framework is illustrated by numerical examples, including the optimal control of elastic geodesics and surfaces of minimal energy by force fields and the control of inextensible rods and shells.
\end{abstract}

\section{Introduction}
Optimal control problems governed by partial differential equations arise throughout the sciences and engineering, whenever one seeks to steer the state of a physical system, described by a PDE, toward a desired configuration while keeping the cost of the applied control small. The classical theory for such problems in \emph{linear spaces}, is by now well established: existence of optimal controls, derivation of first-order (KKT) and second-order optimality conditions, and a rich array of numerical solution methods, are available and well understood (cf., e.g., \cite{troltzsch2010optimal,hinze2008optimization}).

A great variety of problems, however, involve \emph{geometric variational equations}, i.e., PDEs whose solutions are mappings into a nonlinear manifold. Such problems are encountered frequently in continuum mechanics, or problems of stationary action \cite{MarsdenRaitu:1994}, in elasticity theory and geometry processing \cite{heeren2012time}, in image and shape analysis \cite{weinmann2014total}, in liquid crystal modeling \cite{hardt1986existence}, and in micromagnetics \cite{desimone2002reduced}. For an account on the numerical analysis see also \cite{bartels2015numerical,hardering2020geometric}. 
Formulating and solving optimal control problems for states governed by such geometric variational equations is therefore of considerable practical interest. For instance, one may wish to actuate an elastic rod or shell into a target shape by an external force field. 

Geometric optimal control has classically addressed finite-dimensional systems governed by ordinary differential equations on manifolds, formulating the Pontryagin maximum principle intrinsically via Lie brackets and Hamiltonian flows \cite{agrachev2013control,schattler2012geometric}, and more recently used the Riemannian metric itself as a control design tool \cite{goodman2026riemannian}. In the infinite-dimensional setting, a Riemannian view has proven fruitful in shape optimization, where the shape itself is treated as a point on a manifold and Lagrange–Newton methods have been developed accordingly \cite{Schulz2015}. Closely related, equality-constrained optimization on manifolds has been extended to SQP-type methods on Hilbert manifolds \cite{SchielaOrtiz:2021} and to manifold-valued constraints \cite{bergmann2022first}.

Extending PDE-constrained optimization to the geometric setting is, however, not a straightforward exercise. The absence of a global linear structure on the state space means that many of the standard tools from calculus, used in optimal control, cannot be applied verbatim. Instead, one has to introduce concepts and techniques from differential geometry. It turns out that the state and the adjoint state are coupled together in a subtle way. Nevertheless, one would like to preserve, as far as possible, the structural elegance of the classical Lagrangian approach.

In this paper we develop a framework for geometric optimal control problems constrained by variational equations on manifolds. We formulate a corresponding Lagrangian function, derive the resulting first-order optimality conditions, comprising the state equation, an adjoint equation, and a stationarity equation. Building on this optimality system, we devise a Lagrange–Newton method for its numerical solution. The main difficulty is to find a concise and reliable calculus to compute the Newton linear operator, including the quantities, related to the curvature of the underlying manifold.  All quantities are derived in weak form, allowing for implementations in terms of manifold valued finite elements. 
We illustrate the theoretical results on a range of examples of geometric optimal control problems. 

The remainder of the paper is organized as follows. After a recapitulation of the necessary basic concepts of differential geometry in Section 2, we discuss the notion of geometric variational problems and optimal control problems on manifolds in Section 3. For the analysis of this class of problems we require vector transports, as well as induced connections and covariant derivatives on vector bundles, which we recall  in Section 4. In Section 5 we then define an appropriate Lagrangian function, introduce a local splitting technique, and derive first-order optimality conditions corresponding to our geometric optimal control problem. Section 6 presents a Lagrange–Newton method for the numerical solution of this optimality system. We introduce an implementable framework and show that the presented algorithm is equivalent to a Newton method applied to the derivative of the Lagrangian. Section 7 illustrates the theoretical results by a couple of numerical experiments, including the control of elastic geodesics, surfaces of minimal energy, and inextensible rods and shells, before we finally give an outlook on future research directions in Section 8.
\section{Preliminaries}

We consider a \emph{Banach manifold} $\X$, i.e., a topological space $\X$ together with a collection of charts $(U,\phi)$, where each chart $\phi : U \to \mathbb X$ maps an open subset $U\subset\X$ homeomorphically onto an open subset of a real Banach space $\mathbb X$ with norm $\|\cdot\|_{\mathbb X}$. Unless otherwise noted, we assume $\X$ to be of class $C^2$, i.e.\ all chart transition mappings $t_{ji} := \phi_j\circ\phi_i^{-1}$ are $C^2$-diffeomorphisms on their domains. Mappings $F:\X\to\Y$ between two such manifolds are called continuous, differentiable, or of class $C^k$, if their representations in charts have the respective property. It can readily be shown that these properties are independent of the choice of charts.
For a manifold $\X$ we define tangent spaces by using charts, see, e.g., \cite[II §2]{Lang}. Let $x\in\X$ and consider the set 
\[ \lbrace (\phi,v) \; \vert \; \phi:U\rightarrow \mathbb{X} \text{ chart at } x\in\X, \ v\in\mathbb{X} \rbrace. \]
The relation $(\phi_i,v_i)\sim_x(\phi_j,v_j) \ :\Leftrightarrow \ t_{ji}^\prime(\phi_i(x))v_i = v_j$ defines an equivalence relation. The corresponding equivalence class is called a \emph{tangent vector} $v$ of $\X$ at $x$. The set of these equivalence classes at a point $x\in\X$ is a vector space, called the \emph{tangent space} $T_x\X$ of $\X$ at $x$.  Their union forms the \emph{tangent bundle} $\pi : T\X \to \X$, with fibres $\pi^{-1}(x) = T_x\X$, which is itself a $C^1$-manifold. For the tangent bundle $\pi: T\X \to \X$ of a manifold $\X$ we have charts in the following natural way:
Let $(U_i,\phi_i)$ be a chart of $\X$ at $x$. Then, every tangent vector $v\in T_x\X$ has a chart representation $v_i \in \mathbb{X}$. Using this representation, we obtain a \emph{tangent bundle chart} given by
    \begin{align*}
        \Phi_{T\X} : \pi^{-1}(U_i) &\to \mathbb{X} \times \mathbb{X}, \quad 
        (x,v) \mapsto (\phi_i(x), v_i).
    \end{align*}
Every differentiable map $F:\X\to\Y$ induces a linear \emph{tangent map}, also called \emph{derivative}, $F'(x) : T_x\X \to T_{F(x)}\Y$ at each $x\in\X$. A $C^1$-\textit{retraction}, a popular concept, used widely in numerical methods on manifolds (cf., e.g., \cite[Chap. 4]{absil2008optimization} or \cite{boumal2023intromanifolds}), is a $C^1$-mapping $R:T\X \rightarrow \X$, where $R_x : T_x\X\to \X$ for a fixed $x\in \X$ satisfies $R_x(0_x)=x$ and $R_x^\prime(0_x)=Id_{T_x\X}$.

\paragraph{Vector bundles.}
The tangent bundle is the most prominent example for a \emph{vector bundle}. A general vector bundle can be seen as a manifold with special structure. It consists of a \emph{total space} $\E$, a \emph{base manifold} $\Y$ and a \emph{base point projection} $\pi:\E\to\Y$. Unless otherwise noted, we assume that it is of class $C^1$. Its projection is a surjective $C^1$-map which assigns to each $e\in \E$ its base point $y=\pi(e) \in \Y$ on the base manifold $\Y$ \cite[III §1]{Lang}. Associated with any point $y \in \Y$ of the base manifold is a linear space $E_y := \pi^{-1}(y)$,
which is called the \emph{fiber} over $y$. It is a Banachable space, which means that it is a complete topological vector space whose topology is induced by some norm, rendering $E_y$ complete. The zero element in $E_y$ is denoted by $0_y$. The total space $\E$ is given by the union of all fibers $E_y$. Thus, every $e\in \E$ can be written as a pair $e=(y,p)$ with $y=\pi(e)$ and $p\in E_y$. 

Since a vector bundle is also a manifold it can be described via general local charts. However, usually \emph{local trivializations} are used instead, which reflect their structure better. For two Banach spaces $X$ and $Y$, we denote by $L(X,Y)$ the vector space of all continuous linear mappings from $X$ to $Y$. For a Banach space $(\mathbb E,\|\cdot\|_{\mathbb E})$ and an open set $U \subset \Y$, a local trivialization is a diffeomorphism $\tau : \pi^{-1}(U) \to U \times \mathbb E$, such that $\tau_y:=\tau|_{E_y} \in L(E_y,\mathbb E)$ is a linear isomorphism. Thus, any element of $e\in \E$ can be represented as a pair $(y,e_\tau)\in \Y\times \mathbb E$ with $y=\pi(e)$ and $e_\tau$ depending on the trivialization. This is, for example, useful if we want to compare elements of different fibres. In a condensed notation, we sometimes write $e\sim (y, e_\tau)$, which denotes the representation of $e$ in trivializations. Using the transition mapping $T(y) := \tilde{\tau}_y\tau_y^{-1} \in L(\mathbb E,\mathbb E)$ for two trivilizations, we obtain $e_{\tilde \tau} = T(y)e_\tau$. Note that the mapping $T$ is nonlinear in $y$, but linear in $e$ reflecting the linear structure of the fibres.
\begin{example}
    \label{rem:VBCharts}
    For the tangent bundle $\pi: T\X \to \X$ of a manifold $\X$ we can define trivializations naturally by setting $\tau_{x,i}(v):=v_i$, where $v_i$ is the chart representation of $v \in T_x\X$ in the chart $(U_i, \phi_i)$ of $\X$ at $x$. Since tangent vectors transform by $t_{ji}^\prime(\phi_i(x))v_{i} = v_{j}$ under change of charts, the transition mapping is given by $T(x)=t_{ji}^\prime(\phi_i(x))$.
\end{example}
A differentiable mapping $v : \Y \to \E$, such that $\pi(v(y))=y$ for all $y\in \Y$ is called \textit{section} of $\E$. The set of all sections is denoted by $\Gamma(\E)$. Sections of the tangent bundle are called \emph{vector fields}.
The \emph{cotangent space} $T_x\X^* := L(T_x\X,\R)$, where $L(T_x\X,\R)$ denotes the linear functionals on $T_x\X$, is the dual space of $T_x\X$. The corresponding \emph{cotangent bundle} $\pi^* : T^*\X \to \X$ has fibres $T_x\X^*$. Its sections are called \emph{covector fields} (or 1-forms). 
\begin{example}
    Consider a continuously differentiable function $f: \X \to \R$, where $\X$ is a manifold. Then, its derivative $f'$ is a section of the cotangent bundle $T^*\X$, associating to each $x\in \X$ the linearization $f'(x) \in L(T_x\mathcal X, \R)$.
\end{example}
This construction is also valid for \emph{dual vector bundles} $\DualBundleProjection:\E^*\to\Y$ where the fibres $E_y^*$ are defined as the dual spaces of the fibres $E_y$ of a given primal bundle $\pi : \E\to\Y$. 
If $\tau_\E\colon \E \to \mathbb E$ is a trivialization of $\E$ with $e_\tau=\tau_{\E,y}(e)$ for $e\in E_y$, then we obtain a local trivialization $\tau_{\E^*}\colon \E^*\to \mathbb E^*$ of the dual bundle $\E^*$ by the mapping $l_\tau=\tau_{\E^*}(l) \coloneqq l \circ \tau_{\E,y}^{-1}$, such that 
\begin{equation}
    \label{eq:leTriv}
    l_\tau(e_\tau)=l(e).
\end{equation}
We denote the \emph{dual pairing} for $\ell\in \Gamma(\E^*)$ and $e\in\Gamma(\E)$ by $\ell(e)\colon \Y \rightarrow \R, \; y \mapsto \ell(y)e(y) \in \R$.
\paragraph{Fibrewise linear mappings.}
A section $\ell \in \Gamma(\E^*)$ is a mapping $\ell : \Y \to \E^*$ assigning to each $y \in \Y$ a linear functional $\ell(y) \in E_y^*=L(E_y,\R)$. The concept of a \emph{fibrewise linear mapping} generalizes this by replacing the codomain $\mathbb{R}$ of each linear functional with the fibre of another vector bundle $\tilde{\pi}:\Y\to\tilde{\E}$. To this end, we consider for every $y\in\Y$ the set of continuous linear mappings $L(E_y,\tilde{E}_y)$ and collect all these in a \emph{vector bundle} $\L(\E,\tilde{\E})$, i.e., 
\[
    \L(\E,\tilde{\E}):= \{(y, l) \mid y\in \Y \text{ and } l \in L(E_{y}, \tilde{E}_y)\},
\]
with base point projection $\pi_\L : \L(\E,\tilde{\E}) \to \Y$, $(y, l) \mapsto y$ for $y\in \Y$ and fibres $L(E_{y}, \tilde{E}_y)$. 
A \emph{fibrewise linear mapping} is a section $S \in \Gamma(\L(\E,\tilde{\E}))$ of the vector bundle $\pi_\L : \L(\E,\tilde{\E}) \to \Y$, i.e,
\[
    S: \Y \to \L(\E,\tilde{\E}) \text{ with } \pi_\L \circ S=\mathrm{id}_{\Y}. 
\]
In particular, $S(y) : E_{y} \to \tilde{E}_y$ is a linear mapping for every $y\in \Y$.
From a fibrewise linear mapping $S$ we obtain a continuous mapping
\[
    \NE{S} : \E \to \tilde{\E}, \qquad v \mapsto \NE{S}(v) := S(\pi(v))\,v.
\]
We call a fibrewise linear mapping $S$ \emph{locally bounded} or \emph{differentiable}, if its representation in local trivializations
\begin{equation}\label{eq:FibrewiseTriv}
S_\tau : U \to L(\mathbb E,\tilde{\mathbb E}), \quad S_\tau(y) = \tau_{\L(\E, \tilde{\E})}(S) := \tau_{\tilde{\E},y}\circ S(y) \circ \tau_{\E,y}^{-1} \in L(\mathbb E,\tilde{\mathbb E})
\end{equation}
has the respective property.  
Then, $\NE{S}'$ satisfies, by the product rule, in trivializations $(y,e_\tau)\in \Y\times\mathbb E$ for $e\in E_y$
\begin{equation}
    \label{eq:FibrewiseLinMappingTrivializationDerivative}
    \NE{S}'(e)\delta e \sim \NE{S}_\tau'(y,e_\tau)(\delta y, \delta e_\tau) = (S_\tau'(y)\delta y) e_\tau + S_\tau(y)\delta e_\tau  \quad \text{ for all } (\delta y,\delta e_\tau) \in T_{y}\Y\times \mathbb E.
\end{equation}
Here, $\pi'(e)\delta e = \delta y$ for $\delta e \in T_e\E$.
\section{Variational problems and optimal control on manifolds}\label{sec:VariationalProblemsAndOCP}
Variational problems are at the heart of the analytical and numerical treatment of large classes of partial differential equations (PDEs). 
As a simple example on linear spaces, consider the following minimization problem on a bounded domain $\Omega\subset\R^2$:
\[
    \min_{y\in Y} \int_\Omega \frac12 \Vert y'(\omega) \Vert_F^2 - \langle f(\omega), y(\omega)\rangle_2 \, d\omega.
\]
Here, $Y$ is the Sobolev space $H_0^1(\Omega, \R^3)$ of $L_2(\Omega)$-mappings into $\R^3$ with weak derivatives in $L_2(\Omega)$ and Dirichlet boundary conditions, $f \in L_2(\Omega, \R^3)$, $\Vert \cdot \Vert_F$ denotes the Frobenius norm, and $\langle \cdot, \cdot \rangle_2$ denotes the Euclidean inner product. A solution of this problem describes an \emph{elastic surface of minimal energy} in $\R^3$ under the force field given by $f$. By differentiation with respect to $y\in Y$ in a direction $v\in H^1_0(\Omega, \R^3)$ we get the first-order optimality condition
\begin{equation}\label{eq:minimalSurfaceR3}
    \int_\Omega \langle y'(\omega), v'(\omega)\rangle_F - \langle f(\omega), v(\omega)\rangle_2  \, d\omega = 0 \quad \text{ for all } v\in Y,
\end{equation}
which is a \emph{variational equation} and the \emph{test functions} $v$ can be seen as infinitesimal perturbations of the solution $y$. 
The first and the second part of this integral may be viewed as a (linear) mapping $A$ and as an element $f$ of $Y^*$ defined by
\[
    A : Y\to Y^*, \; A(y)v := \int_\Omega \langle y'(\omega), v'(\omega)\rangle_F \,d\omega,  \qquad f\in Y^*, \;f\,v := \int_\Omega \langle f(\omega), v(\omega)\rangle_2  \, d\omega.
\]
Equation \eqref{eq:minimalSurfaceR3} then can be written as $A(y) = f \in Y^*$, or equivalently, $A(y)v = f \, v$ tested against all $v \in Y$. In general (for example in parabolic PDEs), the space of test functions, called $P$, can differ from $Y$, $f\in P^*$ and $A : Y \to P^*$ can be nonlinear. In this setting we obtain the variational equation
\[
 A(y)=f \text{ in } P^*  \quad \Leftrightarrow \quad A(y)v=f\, v \quad \text{for all } v\in P. 
\]
Finite elements can be used to solve these problem numerically, restructing the infinite dimensional function spaces $Y$ and $P$ by finite dimensional counterparts. 

In many applications, such as the theory of harmonic mappings \cite{helein2008harmonic}, continuum mechanics \cite{bartels2015numerical, MarsdenRaitu:1994}, and micromagnetism \cite{desimone2002reduced}, \emph{geometric variational problems} arise. Here, the \emph{state} $y$ is no longer valued in a fixed linear space, but in a manifold $\M$, e.g., when $y$ is constrained to take values on a sphere, a Lie group, or more generally a manifold of admissible states. 

As an illustration, we consider \eqref{eq:minimalSurfaceR3}, which is however now confined to an embedded submanifold $\M \subset \mathbb \R^3$. This means that we look for solutions $y \in \Y := H^1(\Omega,\M)$. This \emph{nonlinear} function space consists of all elements of $H^1(\Omega,\mathbb R^3)$ with values in $\M$ almost everywhere (cf., e.g.,  \cite{hajlasz2009sobolev}). In this setting, admissible perturbations of $y$ naturally live in the tangent space $T_y\Y$, which changes with $y$. Thus, we get the equation
\begin{equation}\label{eq:simplegeovar}
\int_\Omega \langle y'(\omega), v'(\omega)\rangle_F - \langle f(\omega), v(\omega)\rangle_2\; d\omega = 0 \quad \text{ for all } v \in T_y\Y.
\end{equation}
Here,  the linear space $T_y \mathcal Y$ of test functions can be characterized as follows:
\[
 T_y \mathcal Y = \{ v \in H^1(\Omega,\mathbb R^3) : v(\omega) \in T_{y(\omega)} \M \; \text{for almost every $\omega \in \Omega$} \}. 
\]
In contrast to the case of linear spaces, we now obtain a (fibrewise linear) mapping $A : \mathcal Y \to T^*\mathcal Y$ into the cotangent bundle $T^*\mathcal{Y}$ and a section $f:\mathcal Y \to T^*\mathcal Y$. Crucially, the space of test functions depends on $y$. Consequently, if we perturb $y$, we have to take into account that the test function has to be perturbed as well. A solution of the equation $A(y)=f(y)$ then yields a surface of minimal energy confined to the sphere and creating a balance of forces. Just as in the linear case, the vector bundle $\pi:\E \to \Y$ of test functions can differ from $T\Y$ in general, so that with $A\in \Gamma(\E^*)$ and $f\in\Gamma(\E^*)$ we obtain the geometric variational equation:
\[
 A(y)=f(y) \text{ in } E_y^* \quad \Leftrightarrow \quad A(y)v=f(y)v \quad \text{for all } v\in E_y. 
\]
Obviously, these problems are inherently nonlinear, but finite element methods have been developed for this class \cite{bartels2015numerical,hardering2020geometric}. In \cite{weigl2025newton} a Newton method has been proposed that can be applied as a solution algorithm to such geometric variational problems. 
\paragraph{PDE-constrained optimization.}
Instead of considering a given force field $f$, in PDE-constrained optimization an unknown force field, called the \emph{control} is computed, such that if the control is applied within the variational problem a desired configuration of the PDE solution is achieved. In our example, we use a control $u \in U := L_2(\Omega, \R^3)$ 
to act on the \emph{state} $y$. 
The objective is usually to find an optimal balance between control cost and deviation of the controlled state from the desired configuration. 
Classically, on linear spaces, PDE-constrained optimization problems can often be formulated in abstract operator form as
\begin{equation}\label{eq:classicOptControl}
\begin{aligned}
    &\min_{(y,u)\in Y\times U} J(y,u) 
     \hspace{0.4cm} \mathrm{s.t.} \; A(y) - Bu = 0 \text{ in } P^*, 
\end{aligned}
\end{equation}
where $Y, \; U$ and $P$ are linear reflexive spaces, $B:U\to P^*$ is linear, $A:Y\to P^*$ is a (possibly nonlinear) operator, modelling the weak form of a PDE with the space of test functions $P$, and $J : Y\times U \to \R$ is the objective function (cf., e.g., \cite{troltzsch2010optimal,hinze2008optimization}).

Let us extend this problem class to a setting, where the PDE is given by a geometric variational equation and the state is a mapping into a nonlinear manifold $\M$.
To this end, let $\Y$ be a Banach manifold of mappings into $\M$. Consider a vector bundle  $\pi:\E\to\Y$ and its dual bundle $\DualBundleProjection:\E^*\to\Y$. $\E$ will play the role of the space $P$ in the abstract form of the optimal control problem. In our examples, the vector bundle $\E$ will usually be the tangent bundle of $\Y$. 
However, we think that it is also in this case beneficial to distinguish between perturbations of $y\in\Y$ and test functions of the variational equation. 
Let $U$ be a linear space. Similarly to linear spaces we write the variational equation in operator form. For this, we consider a $C^2$-section $A \in \Gamma(\E^*)$, i.e., for every $y\in\Y$ we have a linear mapping $A(y) : E_y \to \R$, and a fibrewise linear $C^2$-mapping $B \in \Gamma(\L(\Y\times U, \E^*))$. This means that $B(y) : U\to E_y^*$ is a linear mapping for every $y\in \Y$. This yields the \emph{state equation}
\[
    A(y) - B(y)u = 0_y^* \quad \text{ in } E_y^*
\]
which is a geometric variational problem. In contrast to linear spaces, $B$ has to depend on $y$ as well to reflect that we test the equation with matching test functions in $E_y$. For a given objective function $J : \Y\times U \to \R$ we consider the geometric PDE-constrained optimization problem, also called \emph{geometric optimal control problem},
\begin{equation}\tag{GOCP}\label{eq:OCP}
\begin{aligned}
    &\min_{(y,u)\in \Y\times U} J(y,u) \hspace{0.4cm}
     \mathrm{s.t.} \; A(y) - B(y)u = 0_y^* \quad \mathrm{in} \; E_y^*.
\end{aligned}
\end{equation}
The manifold $\Y$ is called the \emph{state space}, $U$ is called the \emph{control space}, $y\in \Y$ is called \emph{state} and $u\in U$ is the \emph{control} that acts on the state. Some concrete examples are provided in Section~\ref{sec:Numerics}, below.

\section{Vector transports and connections} 

For a further analysis and development of algorithms for optimal control problems on manifolds we will need \emph{vector transports} and the concept of a \emph{connection} on a vector bundle. Both concepts are fundamental to differential geometry and closely related.

A \emph{vector transport} is a fibrewise linear mapping that can be used to map vectors from one fibre to another. A prominent example in differential geometry is the parallel transport along geodesics, but for many manifolds, alternatives are available, which are more easy to implement and more efficient to compute. 
For our purpose it is convenient to distinguish between two variants:
\begin{definition}
For a vector bundle $\pi:\E\rightarrow\Y$ and $\hat y \in \Y$ we define a \emph{vector forward transport} as a differentiable fibrewise linear mapping $\VTforward{\hat y} \in \Gamma(\L(\Y \times E_{\hat y}, \E))$, i.e.
\[
    \VTforward{\hat y} : \Y \to \L(\Y \times E_{\hat y}, \E), \ \text{with } \pi_\L(\VT{y}{\rightarrow}(y)) = y \; \text{ for all } y\in\Y,
\]
with the properties that $\VTforward{\hat y}(\hat y)=Id_{E_{\hat y}}$ and $\VTforward{\hat y}(y)$ is continuously invertible for all $y\in \Y$. \\
We define a \emph{vector back-transport} as a differentiable fibrewise linear mapping $\VTback{\hat y}\in \Gamma(\L(\E,\Y \times E_{\hat y}))$, i.e. 
\[
    \VTback{\hat y} : \Y \to \L(\E, \Y \times E_{\hat y}), \ \text{with } \pi_\L(\VT{\hat y}{\leftarrow}(y)) = y \;\text{ for all } y\in\Y
\]
with the properties that $\VTback{\hat y}(\hat y)=Id_{E_{\hat y}}$ and $\VTback{\hat y}(y)$ is continuously invertible for all $y\in \Y$. 
\end{definition}

\begin{figure}[h]
    \centering
    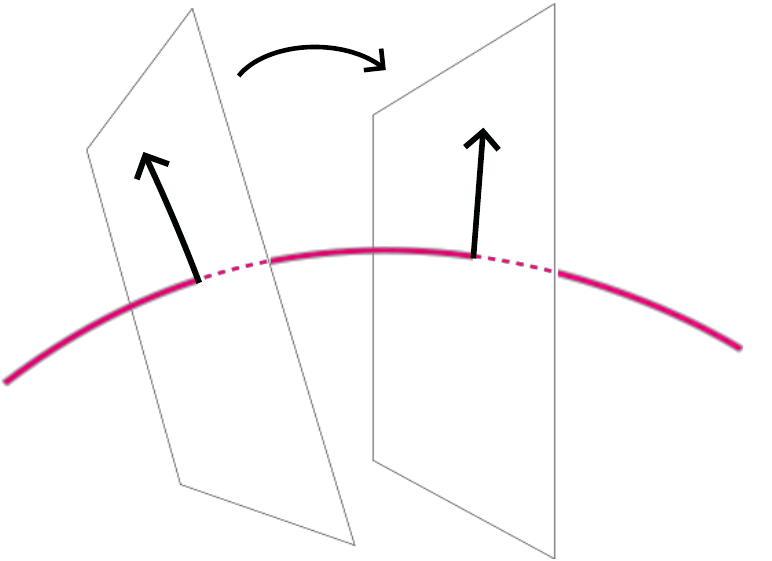
    \caption[Sketch of a vector forward transport]{Sketch of a vector forward transport $\VT{\hat y}{\rightarrow} \in \Gamma(\L(\Y\times E_{\hat y}, \E))$ with~$\VT{\hat y}{\rightarrow}(y) \in L(E_{\hat y}, E_{y})$.}\label{fig:VTforward}
\end{figure}

A visualization of a vector forward transport is given in
Figure~\ref{fig:VTforward}.
By restricting $\Y$ appropriately, vector transports can also be defined locally, which is enough for our purposes throughout the paper.
Given a vector forward transport $\VT{\hat y}{\rightarrow} \in \Gamma(\L(\Y\times E_{\hat y}, \E))$, we can derive a vector back-transport $\VT{\hat y}{\leftarrow}\in \Gamma(\L(\E,\Y \times E_{\hat y}))$ by taking fibrewise inverses $\VT{\hat y}{\leftarrow}(y) := (\VT{\hat y}{\rightarrow}(y))^{-1}$, $y\in\Y$, and vice versa. 

From a given vector forward transport $\VTforward{\hat y} \in \Gamma(\L(\Y\times E_y, \E))$ on $\E$, we define a \emph{covector back-transport} $\VTstar{\hat y}\in \Gamma(\L(\E^*, \Y\times E_{\hat y}^*))$, which acts as a vector back-transport on the dual bundle $\E^*$, by taking fibrewise \emph{adjoints}
    \[
        \VTstar{\hat y}(y) : E_{y}^* \to E_{\hat y}^* \quad \text{ with } \quad \VTstar{\hat y}(y) l :=\left(\VTforward{\hat y}(y)\right)^*l := l \circ \VTforward{\hat y}(y) \quad \text{ for } y\in \Y \text{ and } l\in E_{y}^*. 
    \]

\paragraph{Vector transports as local splittings.} 
Consider a vector back-transport $\VTback{\hat{y}} \in \Gamma(\L(\E, \Y\times E_{\hat{y}}))$ for a fixed $\hat{y}\in \Y$. 
Then, we obtain a \emph{local splitting} of the vector bundle $\pi:\E\to\Y$ around $\hat{y}$ by using  
\begin{equation}\label{eq:splittingE}
    \NE{\VTback{\hat{y}}}:\E \to \Y\times E_{\hat{y}} \quad \text{ with }  \NE{\VTback{\hat{y}}}(e) = (\pi(e), \VTback{\hat{y}}(\pi(e))e) \text{ for } e\in E_{\pi(e)}\subset \E.
\end{equation}
Furthermore, for $e = (y,p)\in \Y\times E_y$ and a tangent vector $\delta e \in T_e\E$ we obtain a \emph{splitting} of $\delta e$ into $(\delta y, \delta p)\in T_{y}\Y \times E_{\hat{y}}$ by applying
\[
    \NE{\VTback{\hat{y}}}'(e) : T_e\E \to T_y\Y\times E_{\hat{y}} \text{ with } \NE{\VTback{\hat{y}}}'(e)\delta e = (\delta y, \delta p).
\] 
Here, $\delta y = \pi'(e)\delta e$ is the \emph{tangential part} and $\delta p$ is the \emph{fibre part}, which by \eqref{eq:FibrewiseLinMappingTrivializationDerivative} depends on the trivialization: $\delta p_\tau=\VTback{\hat y,\tau}(y)\delta e_\tau+\big(\VTprime{\hat y,\tau}{\leftarrow}(y)\delta y\big)e_\tau $.

For a function $f : \E\to \R$ we obtain a \emph{locally splitted model} of $f$ by the composition with a vector forward transport:
\begin{equation}\label{eq:localModel}
    f_{\hat{y}} := f \circ \NE{\VTforward{\hat{y}}}: \Y \times E_{\hat{y}} \to \R
\end{equation}
with $f_{\hat{y}}(\hat{y},e) = f(e)$ for $e\in E_{\hat{y}}$. This will be particularly useful for the derivation of optimality conditions and an implementable Lagrange-Newton method. 
By the chain rule, applied to \eqref{eq:localModel} we obtain:
\[
  f_{\hat y}'(y,p)=f'(e)\NE{\VTforward{\hat{y}}}'(y,p) : T_y\Y \times E_{\hat{y}} \to \R, \quad \text{ where } e=\NE{\VTforward{\hat{y}}}(y,p).
\]

Finally, we can \emph{combine} $(\delta y, \delta p) \in T_y\Y \times E_{\hat{y}}$ to a vector $\delta e$ by using the derivative of the vector forward transport, i.e., for $(y, p) \in \Y \times E_{\hat{y}}$:
\[
    \NE{\VTforward{\hat{y}}}'(y,p) : T_y\Y \times E_{\hat{y}} \to  T_e\E \text{ with } e = \NE{\VTforward{\hat y}}(y,p)
\]

\begin{lemma}\label{lem:InvVPrime}
 Let $\VTforward{\hat y}$ be a vector forward transport and denote $\hat e=(\hat y, p) \in \E$. Then 
 \[
 \NE{\VTforward{\hat{y}}}'(\hat y,p) \in L(T_{\hat y}\Y \times E_{\hat{y}},T_{\hat e}\E)
 \]
 is continuously invertible.
\end{lemma}
\begin{proof}
Define the vector back-transport $\VTback{\hat{y}}(y):=(\VTforward{\hat y}(y))^{-1}$ via fibrewise inverses. 
Then, for $(y,p)\in \Y \times E_{\hat{y}}$ we have $\NE{\VTback{\hat{y}}}(\NE{\VTforward{\hat{y}}}(y,p)) = (y,p)$ and $\NE{\VTforward{\hat{y}}}(\NE{\VTback{\hat{y}}}(e)) = e$ for $e\in\E$, at least locally. Thus, we get by the chain rule at $\hat e = (\hat{y}, p)$:
\begin{equation}\label{eq:VTprimeInverse}
    \NE{\VTback{\hat{y}}}'(\NE{\VTforward{\hat{y}}}(\hat y,p))\NE{\VTforward{\hat{y}}}'(\hat y,p)(\delta y, \delta p) = (\delta y, \delta p) \quad \text{ and } \quad \NE{\VTforward{\hat{y}}}'(\NE{\VTback{\hat{y}}}(\hat e))\NE{\VTback{\hat{y}}}'(\hat e)\delta e = \delta e
\end{equation}
for $(\delta y, \delta p) \in T_{\hat{y}}\Y\times E_{\hat{y}}$ and $\delta e \in T_{\hat{e}}\E$.
Thus, we obtain $\NE{\VTback{\hat{y}}}'(\NE{\VTforward{\hat{y}}}(\hat{y},p))\NE{\VTforward{\hat{y}}}'(\hat{y},p) = \id_{T_{\hat{y}}\Y\times E_{\hat{y}}}$ and $\NE{\VTforward{\hat{y}}}'(\NE{\VTback{\hat{y}}}(\hat{e}))\NE{\VTback{\hat{y}}}'(\hat{e}) = \id_{T_{\hat{e}}\E}$, and thus the desired continuous invertibility. 
\end{proof}
\paragraph{Connections, induced by vector transports.}
A \emph{connection} is a fundamental geometric structure imposed on $\pi:\E\to\Y$.
The most prominent example is the \textit{Levi-Civita connection} on $\E=T\X$ for a Riemannian manifold $\X$ (see, e.g., \cite[VIII §4]{Lang}). Connections give rise to further concepts like the covariant derivative, curvature, geodesics, and parallel transports. They are described in the literature in various equivalent ways (cf., e.g., \cite[IV §3 or X §4 or XIII]{Lang}). It is well known that the tangent space $T_e\E$ of a vector bundle $\pi:\E\to\Y$ at $e\in\E$ has a canonical subspace, the \emph{vertical space}
$\mathrm{Vert}_e=\ker\pi'(e)\cong E_{\pi(e)}$, but no canonical complement (cf. \cite[IV §3]{Lang}). A \emph{connection} on $\E$
resolves this by choosing a complement, the \emph{horizontal space} $\mathrm{Hor}_e$. Equivalently, one can choose a \emph{connection map} $Q_e:T_e\E\to E_{\pi(e)}$ at $e\in\E$ that projects $T_e\E$ onto the fibre over $e$, and set $\mathrm{Hor}_e:=\ker Q_e$. In this paper, we choose this formulation because we want to use the connection not only as a theoretical concept, but as an algorithmic tool, provided that $Q_e$ can be evaluated efficiently. 

The insight that the choice of a connection on $\E$ describes in an infinitesimal way how neighboring fibres are related gives rise to the idea of utilizing vector transports as the key ingredient for achieving an implementable framework. Thus, we consider a vector back-transport $\VTback{y}$ for fixed $y \in \Y$ and its induced \emph{local splitting} $\NE{\VTback{y}}: \E\to \Y\times E_{y}$ via \eqref{eq:splittingE} with $\hat{y}=y$. A connection is obtained from this splitting by viewing it \emph{infinitesimally} at a point $e \in E_{y}$:
\[
    \NE{\VTback{y}}'(e) : T_{e}\E \to T_{y}\Y\times E_{y} \text{ with } \NE{\VTback{y}}'(e)\delta e = (\delta y, \delta p).
\] 
For $\delta e\in T_{e}\E$, the tangential part 
$\delta y$ is canonically determined by $\delta y = \pi'(e)\delta e$. The fibre part is given by
\[
\delta p=\mathrm{pr}_2\left(\NE{\VTback{y}}'(e)\,\delta e\right)\in E_{y},
\]
 where $\mathrm{pr}_2:T_{y}\Y\times E_{y}\to E_{y}$ denotes the projection onto the second component.
For $\delta e\in\mathrm{Vert}_{e}=\ker\pi'(e)\cong E_{y}$ we have
$\delta y=0$, and in trivializations, according to \eqref{eq:FibrewiseLinMappingTrivializationDerivative},
\[
  \mathrm{pr}_2\left(\NE{\VTback{y}}'(e)\delta e\right)\sim\delta e_\tau+(\VTprime{y,\tau}{\leftarrow}(y)\delta y)\,e_\tau=\delta e_\tau .
\]
Hence, $\mathrm{pr}_2\circ\NE{\VTback{y}}'(e)$ restricted to $\mathrm{Vert}_{e}\cong E_{\hat{y}}$ is the identity,
i.e. it is a projection onto the fibre, as required of a connection map.
This yields an \emph{induced} connection $Q$ on $\E$ defined via $\VTback{y}$ (cf. \cite[Lemma 2.2]{weigl2024newton}). 
\begin{definition}
\label{def:VTconnection}
Let $\pi:\E\to\Y$ be a vector bundle, $\VT{y}{\leftarrow} \in  \Gamma(\L(\E,\Y \times E_{y}))$ be a vector back-transport and $e\in E_y$, $y= \pi(e)$. We define the \emph{induced connection map} at $e$ by setting
\[
    Q_{e}:= \mathrm{pr}_2\circ\NE{\VT{y}{\leftarrow}}^\prime(e):T_{e}\E \to E_{y}.
\]
For $\delta e \in T_e\E$ with $\delta e \sim (\delta y, \delta e_\tau) \in T_y\Y \times \mathbb{E}$ its representation in trivializations is given by
\begin{equation*}
    Q_e\delta e \sim Q_{e,\tau}(\delta y,\delta e_\tau)=\delta e_\tau + (\VTprime{{y,\tau}}{\leftarrow}(y)\delta y) e_\tau.
\end{equation*}
If $\VT{y}{\leftarrow}$ is defined via fibrewise inverses of a vector forward transport $\VT{y}{\rightarrow}\in \Gamma(\L(\Y\times E_y,\E))$, then 
\begin{equation}\label{eq:VTprimeForwardBackward}
(\VTprime{{y,\tau}}{\leftarrow}(y)\delta y) e_\tau=-(\VTprime{{y,\tau}}{\rightarrow}(y)\delta y) e_\tau.
\end{equation}
\end{definition}
A connection on the tangent bundle $T\Y$ is called \emph{torsion-free} if the bilinear mapping $\VTprime{y,\Phi}{\leftarrow}(y) : \mathbb{Y} \times \mathbb{Y} \to \mathbb{Y}$ is symmetric in natural tangent bundle charts. The choice of a connection map on $\E$ gives rise to the definition of the \emph{covariant derivative} of a section $F\in \Gamma(\E)$. 
\begin{definition}[Covariant derivative]\label{def:covariantDerivative}
Let $\pi:\E\to\Y$ be a vector bundle equipped with a connection map $Q_{e}:= \NE{\VT{y}{\leftarrow}}^\prime(e):T_{e}\E \to E_{y}$ for $e\in E_y, \, y=\pi(e)$, induced by a vector back-transport $\VTback{y}$. For a differentiable section $F\in \Gamma(\E)$ we define the \emph{covariant derivative} of $F$ as a section of $\L(T\Y, \E)$, i.e., 
\[
    \nabla^\E F : \Y \to \L(T\Y, \E)\quad \text{ with }\quad 
    \nabla^\E F(y) := Q_{F(y)}\circ F'(y) : T_y\Y \to E_y \quad \text{ for } y\in \Y.
\]
\end{definition}
On the dual bundle $\DualBundleProjection\colon\E^* \to \Y$ we consider a covector back-transport defined via fibrewise adjoints of a vector forward transport $\VTforward{y} \in \Gamma(\L(\Y\times E_{y}, \E))$. For $l \in \E^*$ we define a connection on $\E^*$ according to Definition \ref{def:VTconnection} by setting
\begin{equation*}
Q_l^* := \mathrm{pr}_2 \circ \NE{\VTstar{y}}'(l):T_{l}\E^*  \to E^*_{y}, \;  y=\DualBundleProjection(l).
\end{equation*}
Its representation in trivializations is given by 
 \begin{equation}\label{eq:Qstartriveq}
 (Q^*_l\delta l)(v)=Q^*_{l,\tau}(\delta y,\delta l_\tau)(v_\tau)=\delta l_\tau(v_\tau) + l_\tau\circ (\VTprime{{y,\tau}}{\rightarrow}(y)\delta y) v_\tau
\end{equation}
for all $v\in E_y$, $\delta l \in T_l\E^*$ and $\delta y \in T_y\Y$. Analogously to Definition \ref{def:covariantDerivative}, we define a covariant derivative $\nabla^{\E^*}$ of a section $\ell\in \Gamma(\E^*)$ by setting $\nabla^{\E^*}\ell(y) := Q^*_{\ell(y)}\circ \ell'(y)$ for $y\in \Y$. We obtain the following relation between primal and dual connections (cf. \cite[Prop. 4.2]{weigl2025newton}).
\begin{proposition}\label{pro:PrimalDualConnectionVT}
    Let $\pi:\E\to\Y$ be a vector bundle and $\DualBundleProjection:\E^* \to \Y$ its dual bundle. Consider a vector forward transport $\VTforward{y}\in \Gamma(\L(\Y\times E_y,\E))$ and let $\VTback{y}$ be defined by taking fibrewise inverses. Then, the choice 
    \begin{equation*}
 \mathrm{pr}_2\circ \NE{\VTstar{y}}'(l):T_{l}\E^*  \to E^*_{y},
\end{equation*}   
    for $l\in \E^*$ with $y=\DualBundleProjection(l)$ defines the dual connection map $Q_l^*$ on $\E^*$ corresponding to the connection map $Q_e = \mathrm{pr}_2\circ\NE{\VTback{y}}'(e)$ at $e$ with $y=\pi(e)=\DualBundleProjection(l)$ on $\E$, i.e., the product rule
    \begin{equation}\label{eq:defdualconnection}
        \begin{aligned}
        \ell(e)'(y) \delta y &= (Q_{\ell(y)}^*\circ \ell^\prime(y)\delta y)(e(y)) + \ell(y)(Q_{e(y)}\circ e^\prime(y) \delta y)\\
        &=\left(\nabla^{\E^*} \ell(y)\delta y\right)e(y) + \ell(y)(\nabla^\E e(y)\delta y)
        \end{aligned}
    \end{equation}
    holds for all $e \in \Gamma(\E)$, $\ell \in \Gamma(\E^*)$ and for all $(y,\delta y) \in T\Y$.
\end{proposition}
\paragraph{Practical computation of covariant derivatives.}
The following two Propositions give a practical way to compute covariant derivatives of sections if the connection is given by differentiating a vector transport. 
   \begin{proposition}
    \label{prop:computationPrimalConnection}
    Let $\pi : \E\to \Y$ be a vector bundle and $\VTback{y}\in \Gamma(\L(\E, \Y\times E_y))$ be a vector back-transport. Consider the induced connection map $Q_e = \mathrm{pr}_2\circ\langle{\VTback{y}}\rangle'(e)$ at $e \in E_y, \; y=\pi(e)$. Let $G\in \Gamma(\E)$. Then, for all $y\in\Y$ the covariant derivative of $G$ can be computed by
    \[
        \nabla^\E G(y)\delta y = \frac{d}{d\eta}\left(\VTback{y}(\eta)G(\eta)\right)\Big\vert_{\eta=y} \delta y \quad \text{ for all } \delta y \in T_y\Y.
    \]
    In particular, if $G$ is defined by $G(\eta) = \VTforward{y}(\eta)\phi$ for $\eta \in \Y$, where $\phi \in E_y$ is fixed, and $\VTforward{y}$ is defined via fibrewise inverses of $\VTback{y}$, we obtain
    \[
        \nabla^\E G(y)\delta y = 0 \quad \text{ for all } \delta y \in T_y\Y.
    \]
\end{proposition}
\begin{proof}
    Let $y\in \Y$. We compute, using local trivializations $\tau$ of $\E$ for arbitrary $\delta y \in T_y\Y$ 
    \begin{align*}
        \frac{d}{d\eta}\left(\VTback{y}(\eta)G(\eta)\right)\Big\vert_{\eta=y} \delta x &= \frac{d}{d\eta}\Big(\VTback{y,\tau}(\eta)G_\tau(\eta)\Big)\Big|_{\eta=y}\delta y
        =G_\tau'(y)\delta y+\VTback{y,\tau}'(y)\delta y\left(G_\tau(y)\right)\\
        &\sim Q_{G(y)}\circ G'(y)\delta y = \nabla^\E G(y)\delta y.
    \end{align*}
    Let $\phi\in E_y$ be fixed. Consider $G\in \Gamma(\E)$ defined by $G(\eta) = \VTforward{y}(\eta)\phi$ for $\eta \in \Y$. Then we obtain
        \begin{align*}
        \nabla^\E G(y)\delta y &= \frac{d}{d\eta}\left(\VTback{y}(\eta)\VTforward{y}(\eta)\phi\right)\Big\vert_{\eta=y} \delta y
        = \frac{d}{d\eta}\phi\Big|_{\eta=y}\delta y
        =0 \quad \text{ for all } \delta y \in T_y\Y.
    \end{align*}
\end{proof}
Similarly, we get a result for dual covariant derivatives (cf. \cite[Prop. 4.3]{weigl2025newton}). It is central for the derivation of Newton's method for variational problems on manifolds. 
\begin{proposition}
    \label{prop:computationDualConnection}
    Let $\pi : \E\to \Y$ be a vector bundle and $q : \E^* \to \Y$ its dual bundle. Let $\VTforward{y}\in \Gamma(\L(\Y\times E_y,\E))$ be a vector forward transport and $\VTback{y}^*$ be given by taking fibrewise adjoints. Consider the induced connection map $Q^*_l = \mathrm{pr}_2 \circ \langle{\VTback{y}^*}\rangle'(l)$ at $l\in \E^*$ with $y=q(l)$. Let $F \in \Gamma(\E^*)$. Then, for all $y\in \Y$ the covariant derivative of $F$ can be computed by
   \begin{equation}\label{eq:MvTfT}
    \left(\nabla^{\E^*}F(y)\delta y\right) e =\frac{d}{d\eta}\Big(F(\eta)(\VTforward{y}(\eta)e)\Big)\Big|_{\eta=y}\delta y \quad \text{ for all } e\in E_{y} \text{ and all } \delta y \in T_y\Y.
   \end{equation}
\end{proposition}
\begin{proof}
        Let $y\in \Y$ and $e\in E_{y}$. We compute, using local trivializations $\tau$ of $\E$,~\eqref{eq:leTriv} and the product rule for arbitrary $\delta y\in T_y\Y$:
\begin{align*}
 \frac{d}{d\eta}\Big(F(\eta)(\VTforward{y}(\eta)e)\Big)\Big|_{\eta=y}\delta y &= \frac{d}{d\eta}\Big(F_\tau(\eta)(\VTforward{y,\tau}(\eta)e_\tau)\Big)\Big|_{\eta=y}\delta y = F_\tau'(y)\delta y(e_\tau)+(F_\tau(y)\circ \VTprime{y,\tau}{\rightarrow}(y)\delta y)(e_\tau)\\
 &=Q^*_{F(y),\tau}(\delta y, F_\tau'(y)\delta y)(e_\tau) = (Q^*_{F(y)}\circ F'(y)\delta y)(e) = \left(\nabla^{\E^*}F(y)\delta y\right)e.
\end{align*}
\end{proof}
The idea behind \eqref{eq:MvTfT} is that a perturbation of $F$ must be tested with a transported test function $v(y):=\VTforward{y}(\eta)e$ in order to obtain a well defined dual pairing on the perturbed fibre. 
 We thus call the application of \eqref{eq:MvTfT} the \emph{moving test function technique}.

\section{Lagrange function and optimality conditions}\label{sec:OptimalityConditions}
Throughout the remainder of this paper, we assume that $\E$ and $\E^*$ are $C^2$-vector bundles. Whenever we require a connection on $\E$ or $\E^*$, we understand it to be the one induced by a chosen vector transport according to Definition \ref{def:VTconnection} and Proposition \ref{pro:PrimalDualConnectionVT}. With this, we can now come back to our optimal control problem \eqref{eq:OCP}. In this section, we will define an appropriate Lagrangian function and derive first-order optimality conditions.

To motivate our approach, let us start with a classical abstract optimal control problem \eqref{eq:classicOptControl} in linear spaces to which we apply the \emph{formal Lagrange technique}.
The Lagrangian function corresponding to this optimal control problem is defined by testing the state equation with a Lagrangian multiplier $p \in P$, also called \emph{adjoint state}:
\begin{equation}\label{eq:adjointlinear}
    L : Y\times U \times P \to \R, \quad L(y,u,p) = J(y,u) + (A(y)-Bu)p.
\end{equation}
Obviously, $L$ depends on the triplet of \emph{independent variables} $(y,u,p)\in Y\times U\times P$. Then, at least formally, first-order optimality conditions can be derived by computing the partial derivatives of $L$ with respect to each of these variables and setting them to zero
\[
 0 \overset{!}=L'(y,u,p)=(\partial_y L(y,u,p),\partial_u L(y,u,p), \partial_p L(y,u,p)) : Y\times U \times P \to \R. 
\]
Then, the $3\times 3$ block operator of second partial derivatives yields the Newton linearization for solving the optimality system locally by the well-known Lagrange-Newton method. By the theorem of Schwarz, this operator is symmetric. 

Our aim is to extend this approach to a geometric optimal control problem \eqref{eq:OCP}
\[
   \min_{(y,u)\in \mathcal Y\times U} J(y,u) \text{ s.t. } \left(A(y) - B(y)u\right)v = 0_y \text{ for all } v\in E_y,
\]
where the state equation is a variational problem on a manifold.
Unlike in the linear case, the space of test functions $E_y$ now depends on the current state. Also the Lagrangian has to reflect that the adjoint state serves as a test function in $E_y$. Therefore, there will be a subtle \emph{coupling between the state and the adjoint state}. This is a fundamental difficulty that prevents the direct transfer of the formal Lagrange technique from the linear case to the manifold case.

Let us thus first consider the state and the adjoint state together as one combined variable $e\in \E$ with $y=\pi(e)\in \Y$. This yields the following definition.
\begin{definition}\label{def:Lagrangian}
    We define the \emph{Lagrangian function} corresponding to the optimal control problem \eqref{eq:OCP} as 
    \begin{align*}
        L : \E \times U &\to \R \\
            (e, u) &\mapsto J(\pi(e), u) + (A(\pi(e)) - B(\pi(e))u)(e).
    \end{align*}
    We call $\pi(e) = y$ \emph{state}, $e \in E_{\pi(e)} \subset \E$ \emph{adjoint state} and $u\in U$ \emph{control}. 
\end{definition}
This formulation exhibits a simple structure and reflects the relationship between the state and the adjoint state. However, with this definition it is difficult to derive optimality conditions, directly. In particular, this formulation only allows for a perturbation $\delta e \in T_e\E$ in the combined variable $e$, but not separately in $y$. Thus, the typical structure of optimality systems, consisting of a state equation, an adjoint equation and an optimality condition is lost.
To overcome this difficulty, we introduce a local splitting of $\E$ around a given state $\hat{y} \in \Y$ into a product $\Y \times E_{\hat{y}}$ by the use of a vector transport $\VTforward{y} \in \Gamma(\L(\Y\times E_{y},\E))$ with 
\[
    \NE{\VTforward{\hat{y}}} : \Y\times E_{\hat{y}} \to \E.
\]
According to \eqref{eq:localModel}, this yields a locally splitted version of the Lagrangian function:
\begin{align*}
    L_{\hat{y}} :=  L\circ (\NE{\VTforward{\hat{y}}},  \mathrm{id_U}) : \Y \times E_{\hat{y}} \times U&\to \R, \\
    ((y,p),u) &\mapsto L(\NE{\VTforward{\hat{y}}}(y,p),u) = J(y,u) + \left(A(y)-B(y)u\right)(\VTforward{\hat{y}}(y)p)
\end{align*}
with $L_{\hat{y}}(\hat{y}, p, u) = L(e,u)$ for $e = (\hat{y},p), \; p \in E_{\hat{y}}$.
Our strategy is to compute first and second derivatives of $L_{\hat{y}}$ with respect to the splitted variables and relate them to the derivatives of $L$.

If we differentiate $L_{\hat{y}}$ with respect to $(y, p, u)\in\Y\times E_{\hat{y}}\times U$, we obtain by the chain rule:
\begin{equation}\label{eq:LyDerivativeChainrule}
\begin{aligned}
    L_{\hat{y}}'(y, p, u) = L'(\NE{\VTforward{\hat{y}}}(y,p), u)(\NE{\VTforward{\hat{y}}}'(y,p), \mathrm{id}_U),
\end{aligned}
\end{equation}
which yields for all $(\phi_y, \phi_p, \phi_u) \in T_y\Y \times E_{\hat{y}} \times U$ the relation
\begin{equation*}
        \partial_y L_{\hat{y}}(y,p,u)\phi_y + \partial_p L_{\hat{y}}(y,p,u)\phi_p + \partial_u L_{\hat{y}}(y,p,u)\phi_u = \partial_e L(e, u)\phi_e + \partial_u L(e,u)\phi_u
\end{equation*}
for $e = \NE{\VTforward{\hat{y}}}(y,p)$ and $\phi_e = \NE{\VTforward{\hat{y}}}'(y,p)(\phi_y, \phi_p)$.
Obviously, the left-hand side yields an optimality system in the classical three component form, while in the right-hand side state and adjoint equation cannot be easily decoupled. 
To derive optimality conditions for \eqref{eq:OCP} we require as usual that $L'(e, u) = 0_{(e,u)}^*\in (T_e\E \times U)^*$. Since $\NE{\VTforward{\hat{y}}}'(y,p) \in L(T_y\Y\times E_{\hat{y}}, T_{e}\E)$ is continuously invertible  by Lemma~\ref{eq:VTprimeInverse},
this is equivalent to the three conditions 
\begin{align*}
    \partial_y L_{\hat{y}}(y,p,u) = 0_y^*, \quad \partial_p L_{\hat{y}}(y,p,u) = 0^*_{\hat{y}}, \quad \partial_u L_{\hat{y}}(y,p,u) = 0^*.
\end{align*}
Since $L_{\hat{y}}(\hat{y}, p, u) = L(e,u)$ for $e = (\hat{y},p)$, we thus obtain the optimality conditions by evaluating the partial derivatives at $(\hat{y},p,u)$.
\paragraph{The adjoint equation.}
For the differentiation of $L_{\hat{y}}$ with respect to $y\in\Y$ we observe that $A(\cdot) -B(\cdot)u : \Y\to \E^*$ maps into a dual vector bundle. Moreover, the local model $L_{\hat{y}}$ contains a section $v := \VTforward{\hat{y}}(\cdot)p \in \Gamma(\E)$, which is given by transporting $p \in E_{\hat{y}}$ to the neighboring fibres:
\[
    v : \Y \to \E \quad \text{with} \quad v(y) = \VTforward{\hat{y}}(y)p \in E_y.
\] 
Thus, we have to take covariant derivatives for their linearization. 
\begin{proposition}
Let $\VTforward{y} \in \Gamma(\L(\Y\times E_{y}, \E))$ be a vector forward transport on $\E$ and assume that $\E$ and $\E^*$ are equipped with its induced connections. Then, the partial derivative of the locally splitted Lagrangian $L_{\hat{y}} = L \circ (\NE{\VTforward{\hat{y}}}, \id_U) : \Y \times E_{\hat{y}}\times U \to\R$ with respect to $y$ in the direction $\phi_y \in T_y\Y$ is given by 
 \begin{equation}\label{eq:AdjointEquationAugmented}
    \begin{aligned}
    \partial_y L_{\hat{y}}(y,p,u)\phi_y = \partial_y J(y,u)\phi_y &+ \nabla^{\E^*}(A(\cdot) - B(\cdot)u)(y)\phi_y \; \VTforward{\hat{y}}(y)p \\
    &\quad + (A(y) - B(y)u)\big(\nabla^\E \big(\VTforward{\hat{y}}(\cdot)p\big)(y) \phi_y\big).
    \end{aligned}
\end{equation}
Its representation in trivializations is given by 
\begin{equation}\label{eq:adjointEquationTrivialization}
    \partial_y L_{\hat{y}}(y,p,u)\phi_y = \partial_y J(y,u)\phi_y + (A_\tau'(y) - B'_{\tau}(y)u)\phi_y\,\VTforward{\hat{y}, \tau}(y)p + (A_\tau(y)-B_\tau(y)u)(\VTprime{\hat{y},\tau}{\rightarrow}(y)\phi_y\, p_\tau).
\end{equation}
\end{proposition}
\begin{proof}
Using the product rule for the dual pairing $(A(\cdot)-B(\cdot)u)(\VTforward{\hat{y}}(\cdot)p) :\Y \to \R$ according to Proposition~\ref{pro:PrimalDualConnectionVT}, we get for $\phi_y\in T_y\Y$:
\begin{equation}\label{eq:linearizedStateEq}
    (A(\cdot)-B(\cdot)u)(\VTforward{\hat{y}}(\cdot)p)'(y)\phi_y = \nabla^{\E^*}(A(\cdot) - B(\cdot)u)(y)\phi_y \; v(y) + (A(y) - B(y)u)\big(\nabla^\E v(y) \phi_y\big).
\end{equation}
This directly yields \eqref{eq:AdjointEquationAugmented}.
For a representation, we use a local trivialization $\tau$ on $\E$ and the corresponding trivialization on $\E^*$ and consider $A_\tau(\cdot), \; B_\tau(\cdot)u :\Y\to \mathbb{E}^* \text{ and } v_\tau(\cdot) : \Y \to \mathbb{E}.$
Then, using the representations of the covariant derivatives according to Definition~\ref{def:VTconnection}, \eqref{eq:VTprimeForwardBackward} and \eqref{eq:Qstartriveq} by
\[
    \nabla^{\E} v(y)\phi_y \sim v_\tau'(y)\delta y + \VTprime{y,\tau}{\leftarrow}(y)\phi_y \; v_\tau(y) = v_\tau'(y)\phi_y - \VTprime{{y,\tau}}{\rightarrow}(y)\phi_y \; v_\tau(y),
\] 
and
\[
    (\nabla^{\E^*} (A(\cdot) - B(\cdot)u)(y)\phi_y)e =(A_\tau'(y) - B_{\tau}'(y)u)\phi_y\, e_\tau +  (A_\tau(y)-B_\tau(y)u)(\VTprime{y,\tau}{\rightarrow}(y)\phi_y \; e_\tau)
\]
for $e\in E_y$, we can write a representation of \eqref{eq:linearizedStateEq} as 
\begin{equation}\label{eq:linStateEqTrivialization}
    \begin{aligned}
     &\nabla^{\E^*}(A(\cdot) - B(\cdot)u)(y)\phi_y \, v(y) + (A(y) - B(y)u)\nabla^{\E} v(y)\phi_y \\
     &\quad = (A_\tau'(y) - B'_{\tau}(y)u)\phi_y \, v_\tau(y) +  (A_\tau(y)-B_\tau(y)u)(\VTprime{y,\tau}{\rightarrow}(y)\phi_y \, v_\tau(y)) \\
     &\qquad + (A_\tau(y)-B_\tau(y)u)\left(v_\tau'(y)\phi_y - \VTprime{y,\tau}{\rightarrow}(y)\phi_y \, v_\tau(y)\right) \\
     &\quad = (A_\tau'(y) - B'_{\tau}(y)u)\phi_y\,v_\tau(y) + (A_\tau(y)-B_\tau(y)u)v_\tau'(y)\phi_y.
    \end{aligned}
\end{equation}
Since $v(y) = \VTforward{\hat{y}}(y)p$ we get $v_\tau(y) = \VTforward{\hat{y}, \tau}(y)p$ and $v_\tau'(y)\phi_y = \VTprime{\hat{y}, \tau}{\rightarrow}(y)\phi_y \,p_\tau$. Thus, we obtain the representation
\begin{equation*}
    \partial_y L_{\hat{y}}(y,p,u)\phi_y = \partial_y J(y,u)\phi_y + (A_\tau'(y) - B'_{\tau}(y)u)\phi_y\,\VTforward{\hat{y}, \tau}(y)p + (A_\tau(y)-B_\tau(y)u)(\VTprime{\hat{y},\tau}{\rightarrow}(y)\phi_y\, p_\tau).
\end{equation*}
\end{proof}
\begin{corollary}
    The \emph{adjoint equation} is given by 
    \begin{equation}\label{eq:adjointEquationSimple}
    \partial_y J(\hat{y},u)\phi_y + \left(\nabla^{\E^*}(A(\cdot) - B(\cdot)u)(\hat{y})\phi_y\right) p = 0 \quad \text{ for all } \phi_y \in T_{\hat{y}}\Y.
\end{equation}
\end{corollary}
\begin{proof}
Evaluating \eqref{eq:AdjointEquationAugmented} at $(y, p, u)=(\hat{y},p,u)$ and using that the covariant derivative $\nabla^\E \big(\VTforward{\hat{y}}(\cdot)p\big)(\hat{y})$ vanishes due to Prop.~\ref{prop:computationPrimalConnection}, we obtain the desired equation by setting the derivative $\partial_yL_{\hat{y}}(\hat{y}, p,u)$ to zero.

\end{proof}
\paragraph{The state equation.} We differentiate the locally splitted Lagrangian function $L_{\hat{y}}$ with respect to $p \in E_{\hat{y}}$ while keeping $y$ and $u$ fixed. Using the linearity of $A(y)$ and $B(y)u$ in the test functions, we obtain
\begin{equation}\label{eq:stateEquation}
    \partial_p L_{\hat{y}}(y,u,p)\phi_p = (A(y) - (B(y)u))(\VTforward{\hat{y}}(y)\phi_p) = 0 \text{ for all } \phi_p \in E_{\hat{y}}.
\end{equation}
\paragraph{The stationarity equation.} The stationarity equation equation can be derived by differentiating the Langrangian function with respect to $u \in U$. Using that $B\in \Gamma(\L(\Y \times U, \E^*))$ is a fibrewise linear mapping, i.e., $B(y) : U \to E_y^*$ is a linear mapping for fixed $y\in\Y$, we get
\begin{equation}\label{eq:stationarityEquation}
    \partial_u L_{\hat{y}}(y,u,p)\phi_u= \partial_u J(y,u) \phi_u - (B(y)\phi_u)(\VTforward{\hat{y}}(y)p) = 0 \text{ for all } \phi_u \in U.
\end{equation}
Evaluating the equations \eqref{eq:adjointEquationSimple}, \eqref{eq:stateEquation} and \eqref{eq:stationarityEquation} at $(y,p,u) = (\hat{y}, p, u)$ we obtain the optimality conditions for \eqref{eq:OCP} summarized in the following. 
\begin{innersummary}[Optimality conditions]
    The optimality conditions for our geometric optimal control problem
    \begin{equation*}
\begin{aligned}
    &\min_{(y,u)\in \Y\times U} J(y,u) \\
    & \hspace{0.4cm} \mathrm{s.t.} \; A(y) - B(y)u = 0_y^* \; \mathrm{in} \; E_y^*
\end{aligned}
\end{equation*}
are given by 
\begin{align}
    \partial_y J(y,u)\phi_y + \left(\nabla^{\E^*}(A(\cdot) - B(\cdot)u)(y)\phi_y \right)p &= 0 \quad \text{ for all } \phi_y \in T_y\Y\tag{adjoint equation}\label{eq:adjointEq}\\[0.2cm]
    \partial_u J(y,u) \phi_u - (B(y)\phi_u)p &= 0 \quad \text{ for all } \phi_u \in U\tag{stationarity equation}\label{eq:stationaryEq}\\[0.2cm]
    A(y)\phi_p - (B(y)u)\phi_p &= 0 \quad \text{ for all } \phi_p \in E_y\tag{state equation}\label{eq:stateEq}
\end{align}
\end{innersummary}
\begin{proposition}
    Assume that $(y_\star, u_\star) \in \Y \times U$ is a minimizer of \eqref{eq:OCP}. Assume that the linearized state equation operator
    \[
    \nabla^{\E^*}A(y_\star) : T_{y_\star}\Y \to E_{y_\star}^*
    \]
    is continuously invertible. Then there exists an adjoint state $p \in E_{y_\star}$ such that $(y_\star, p_\star, u_\star)$ is a solution of the optimality system.
\end{proposition}
\begin{proof}
    Working in trivializations the proof proceeds along standard lines. The main step is the application of the inverse function theorem which requires the invertibility assumption (cf. \cite[4.F]{zeidler1986nonlinear} or \cite[Theorem I.5.2]{Lang}).
\end{proof}


\section{Second derivatives and Lagrange-Newton method}
In order to compute a zero of $L' : \E\times U \to T^*\E \times U^*$, and thus a solution of the optimality system, we apply Newton's method. 
Classically, in the case of \emph{linear spaces}, in the $k$-th iteration we would solve the Newton equation
\[
    L''(e_k, u_k)(\delta e_k, \delta u_k) + L'(e_k, u_k) = 0^*,
\]
and afterward compute a Newton step by adding the Newton direction $(\delta e_k, \delta u_k)$ to the current iterate $(e_k, u_k) \in \E \times U$. 

In a geometric setting we observe that the derivative of $L'$ at $(e_k, u_k)$ is a mapping 
\[
    L''(e_k, u_k) : T_{e_k}\E \times U \to T_{L'(e_k,u_k)}(T^*\E \times U^*),
\]
while $L'(e_k, u_k) \in T_{e_k}\E^* \times U^*$. Thus, $L''(e_k, u_k)(\delta e_k, \delta u_k)$ and $L'(e_k, u_k)$ do not fit together, and the classical Newton equation is not well-defined.  To resolve this issue, we apply Newton's method from \cite{weigl2025newton}. According to \cite[Alg. 1]{weigl2025newton} we have to find a connection $Q^{T^*\E}$ on the vector bundle $T^*\E$ which defines a covariant derivative of $L'$ according to Definition \ref{def:covariantDerivative}:
\[
    \nabla^{T^*\E\times U^*} L'(e_k,u_k) : T_{e_k}\E \times U \to T_{e_k}\E^* \times U^*
\]
and in turn a well-defined Newton equation in $T_{e_k}\E^* \times U^*$:
\[
    \nabla^{T^*\E\times U^*} L'(e_k,u_k)(\delta e_k, \delta u_k) + L'(e_k,u_k) = 0_{(e_k,u_k)}^* \in T_{e_k}\E^*\times U^*.
\] 
Using test functions this can equivalently be written as
\[
    \nabla^{T^*\E \times U^*} L'(e_k,u_k)(\delta e_k, \delta u_k)(\phi_e, \phi_u) + L'(e_k,u_k)(\phi_e, \phi_u) = 0 \quad \text{for all } (\phi_e, \phi_u) \in T_{e_k}\E \times U.
\] 
After solving this equation to compute the Newton direction $(\delta e_k, \delta u_k)$, we encounter a second issue: the Newton direction $(\delta e_k, \delta u_k)\in T_{e_k}\E \times U$ cannot be added to the iterate $(e_k,u_k) \in \E\times U$ since they live in different spaces. To solve this problem, we apply a retraction $R$ on $\E$ and an additive update on $U$ to obtain the Newton step
\[
    (e_{k+1}, u_{k+1}) = (R_{e_{k}}(\delta e_k), u_k +\delta u_k).
\]
However, this formulation has the drawback that $\delta e_k$ is a combination of update directions in $y$ and $p$, which is unclear to represent algorithmically. In particular, it is not obvious, how to the construct a (torsion-free) connection on $T\E^* \times U^*$ to define the Newton step and retraction on $T\E$ to perform an update. 

\subsection{Newton step in splitted form}
Our strategy to resolve these issues is to formulate an implementable Newton step in the splitted variables $(\delta y_k,\delta p_k)\in T_{y_k}\Y\times E_{y_k}$ with $y_k = \pi(e_k)$. To this end, we require the following algorithmic ingredients, to be provided by the user at each iterate $y_k\in \mathcal Y$:
\begin{enumerate}
[label=\roman*)]
 \item A $C^1$-retraction $R_{y_k}^\Y : T_{y_k}\Y \to \Y$ on the base manifold $\Y$.
 \item A $C^1$-vector forward-transport $\WTforward{y_k} \in \Gamma(\L(\Y\times T_{y_k}\Y, T\Y))$ on $T\Y$ giving rise to a torsion-free connection. 
 \item A $C^2$-vector forward-transport $\VTforward{y_k} \in \Gamma(\L(\Y\times E_{y_k}, \E))$ on $\E$. 
\end{enumerate}
For many known manifolds these mappings are readily available and, e.g., in the library \texttt{Manifolds.jl} \cite{AxenBaranBergmannRzecki:2023} even some implementions can be found. Derivatives can usually be computed straightforwardly by hand or by automatic differentiation in forward mode. We will discuss some examples in the section on numerical results, below. \\
\;\\
Our idea is again based on using the local splitting of $\E$ into $\Y\times E_{\hat{y}}$ via the vector forward transport $\VTforward{\hat{y}} \in \Gamma(\L(\Y\times E_{\hat{y}},\E))$ on $\E$. Here, we choose $\hat{y}$ as the current state, i.e., $\hat{y} = y_k$.
This yields the locally splitted Lagrangian $L_{y_k} : \Y\times E_{y_k}\times U \to \R$ of $L$ with its derivative 
\[
L_{y_k}' : \Y\times E_{y_k}\times U \to T^*\Y \times E_{y_k}^* \times U^*.
\]
A splitted update direction $(\delta y_k,\delta p_k, \delta u_k)\in T_{y_k}\Y\times E_{y_k}\times U$ for the current iterate $(y_k,p_k,u_k) \in \Y\times E_{y_k}\times U$ is then computed by solving the equation 
\[
    \nabla^{T^*\Y \times E_{y_k}^*\times U^*} L_{y_k}'(y_k, p_k, u_k)(\delta y_k, \delta p_k, u_k) + L_{y_k}'(y_k, p_k, u_k) = 0_{(y_k,y_k,u_k)}^*.
\]  
Since $L_{y_k}'$ is also a mapping into a vector bundle, namely the dual vector bundle $T^*\Y \times E_{y_k}^* \times U^*$, a covariant derivative of $L'_{y_k}$ is required here, which can be computed via the moving test function technique according to Proposition \ref{prop:computationDualConnection}. This requires a vector forward transport $\WWTforward{y_k}$ on the primal vector bundle $T\Y\times E_{y_k} \times U$, defined simply by 
\begin{equation}\label{eq:VTTYEU}
    \WWTforward{y_k}(y,p,u)(\phi_{y_k}, \phi_{p_k}, \phi_{u_k}) := (\WTforward{y_k}(y)\phi_{y_k}, \phi_{p_k}, \phi_{u_k}).
\end{equation}
Then, we compute the covariant derivative of $L_{y_k}'$ by transporting the test function $(\phi_{y_k}, \phi_{p_k}, \phi_{u_k}) \in T_{y_k}\Y \times E_{y_k} \times U$ and differentiating afterwards, i.e.,
\begin{align*}
   &\left( \nabla^{T^*\Y \times E_{y_k}^*\times U^*} L_{y_k}'(y_k, p_k, u_k)(\delta y_k, \delta p_k, u_k)\right)(\phi_{y_k}, \phi_{p_k}, \phi_{u_k}) \\
   & \quad = \frac{d}{d(y,p,u)}\left(L'_{y_k}(y,p,u)\big(\WWTforward{y_k}(y,p,u)(\phi_{y_k}, \phi_{p_k}, \phi_{u_k})\big)\right)\big\vert_{(y,p,u) = (y_k, p_k, u_k)}(\delta y_k, \delta p_k, \delta u_k).
\end{align*}
\begin{remark}
By assumption, the covariant derivative, induced by $\WTforward{y_k}(y)$ is torsion-free, and the above covariant derivative on the product space inherits this property. Since $L_{\hat{y}} : \Y\times E_{\hat{y}}\times U \to \R$ is a $C^2$-function from a manifold into the real numbers, its second derivative $\nabla^{T^*\Y \times E_{y_k}^*\times U^*} L_{y_k}'(y_k, p_k, u_k)$ is symmetric (cf., e.g., \cite[XIII §1]{Lang}). This structural property is beneficial in many ways, when constructing numerical algorithms.
\end{remark}


After computing the direction $(\delta y_k, \delta p_k, \delta u_k)$, we obtain the next iterate by setting 
\begin{equation}\label{eq:NoRetraction}  
    (y_{k+1}, p_{k+1}, u_{k+1}) := (R_{y_k}^\Y(\delta y_k), \VTforward{y_k}(R^\Y_{y_k}(\delta y_k))(p_k + \delta p_k), u_k+\delta u_k), 
\end{equation}
where $R_{y_k}^\Y : T_{y_k}\Y \to \Y$ is a retraction on $\Y$. This yields Algorithm~\ref{alg:ImplementableNewton}.
\begin{algorithm}
    \caption{Lagrange-Newton Method in Splitted Form}\label{alg:ImplementableNewton}
\begin{algorithmic}
\REQUIRE{$(y_0, p_0, u_0) \in \Y \times E_{y_0} \times U$ (initial guess); $\VTforward{y}$ (vector forward transport on $\E$); $\WTforward{y}$ (vector forward transport on $T\Y$), $R_y^\Y$ (retraction on $\Y$)}
\FOR {$k=0,1,2, \dots$}
\STATE build the locally splitted Lagrangian $L_{y_k} = L \circ (\NE{\VTforward{y_k}}, \id_U)$
\STATE $(\delta y_k, \delta p_k, \delta u_k) \leftarrow \text{ solve } \nabla^{T^*\Y \times E_{y_k}^*\times U^*} L_{y_k}'(y_k, p_k, u_k)(\delta y_k, \delta p_k, u_k) + L_{y_k}'(y_k, p_k, u_k) = 0_{(y_k,y_k,u_k)}^*$
\STATE $(y_{k+1}, p_{k+1}, u_{k+1}) = (R_{y_k}^\Y(\delta y_k), \VTforward{y_k}(R^\Y_{y_k}(\delta y_k))(p_k + \delta p_k), u_k+\delta u_k)$
\ENDFOR
\end{algorithmic}
\end{algorithm}

\subsection{Analysis of the Splitted Lagrange Newton method}
At this point it is not obvious that the resulting sequence of iterates corresponds to a true Newton iteration applied to the problem $L'(e,u) = 0_{(e,u)}^*$. First of all, the splitted Lagrangian function $L_{y_k}$ depends on the current iterate. Thus, seemingly, the problem to be solved, changes in each iteration step. Second, the mapping \eqref{eq:NoRetraction} is not a retraction in general as can be readly verified (cf. \eqref{eq:derPseudoRet}, below). 

In the following we show that our iterative step in $(\delta y_k,\delta p_k)$ is \emph{equivalent} to a Newton update, from a Newton direction $\delta e_k$, computed on $T_{e_k}\E$ if the connection on $T^*\E$ and the retraction on $\E$ are chosen suitably. However, neither of these quantities are needed explicitely for the implementation. 

To proceed, let us fix an iterate $\hat y = y_k$ and consider $L:\E\times U \to \R$ and its locally splitted version $L_{\hat{y}} = L \circ (\NE{\VTforward{\hat{y}}}, \id_U) : \Y\times E_{\hat{y}}\times U \to \R$ with their derivatives 
\[
        L' : \E\to T^*\E \quad \text{ and }\quad L_{\hat{y}}' : \Y\times E_{\hat{y}}\times U \to T^*\Y \times E_{\hat{y}}^* \times U^*.
\] 
As we have seen before, the relation $ L_{\hat{y}}'(y, p, u) = L'(\NE{\VTforward{\hat{y}}}(y,p), u)(\NE{\VTforward{\hat{y}}}'(y,p), \mathrm{id}_U)$ for the first derivatives holds. In order to derive a similar relationship for the second covariant derivatives, we first introduce a vector forward transport on $T\E$, which, according to the moving test function technique from Proposition \ref{prop:computationDualConnection}, allows us to compute a covariant derivative of $L'$. To this end, we use the vector forward transport $\VTforward{\hat{y}}$ from the local splitting and a vector forward transport on $T\Y$.
\begin{proposition}\label{prop:VTTE}
    Let $\WTforward{\hat{y}} \in \Gamma(\L(\Y\times T_{\hat{y}}\Y, T\Y))$ be a vector forward transport on $T\Y$ and $\VTforward{\hat{y}} \in \Gamma(\L(\Y\times E_{\hat{y}}, \E))$ be a vector forward transport on $\E$. Then, for $\hat{e}\in \E$
    \begin{equation}\label{eq:VTTE}
        \VVTforward{\hat{e}} : \E \times U \to \L((\E \times T_{\hat{e}}\E)\times U, T\E\times U), \quad \VVTforward{\hat{e}}(e, u)(\phi_e,\phi_u) := \left(\NE{\VTforward{\hat{y}}}'(y,p)(\WTforward{\hat{y}}(y)\phi_y,\phi_p),\phi_u\right),        
    \end{equation}
    where $e= \NE{\VTforward{\hat{y}}}(y,p)$ and $\phi_e = \NE{\VTforward{\hat{y}}}'(\hat{y}, p)(\phi_y, \phi_p)$ for $(\phi_y, \phi_p,\phi_u) \in T_{\hat{y}}\Y\times E_{\hat{y}} \times U$, is a vector forward transport on $T\E\times U$, and it holds 
    \[
        L_{\hat{y}}'(y, p, u)\WWTforward{\hat{y}}(y,p,u)(\phi_{y}, \phi_p, \phi_u) = L'(e,u)\VVTforward{\hat{e}}(e,u)(\phi_e, \phi_u).
    \]
\end{proposition}
\begin{proof}
    Let $\phi_e \in T_{\hat{e}}\E$ with $\phi_e = \NE{\VTforward{\hat{y}}}'(\hat{y}, p)(\phi_y, \phi_p)$ for $(\phi_y, \phi_p)\in T_{\hat{y}}\Y\times E_{\hat{y}}$ and  $\hat{e} = \NE{\VTforward{\hat{y}}}(\hat{y}, p)$. Let $\phi_u\in U$. 
  Then for $e = \NE{\VTforward{\hat{y}}}(y,p) \in \E$ and $u\in U$ we get 
    \begin{equation}\label{eq:VTTEcodomain}
        \VVTforward{\hat{e}}(e,u)(\phi_e, \phi_u)  = (\NE{\VTforward{\hat{y}}}'(y,p)(\WTforward{\hat{y}}(y)\phi_y, \phi_p), \phi_u) \in T_{\NE{\VTforward{\hat{y}}}(y,p)}\E\times U = T_e\E\times U.
    \end{equation}  
    In particular,
    \begin{align*}
        \VVTforward{\hat{e}}(\hat{e},u)(\phi_e, \phi_u) &=
        (\NE{\VTforward{\hat{y}}}'(\hat{y},p)(\WTforward{\hat{y}}(\hat{y})\phi_y,\phi_p), \phi_u)=
        (\NE{\VTforward{\hat{y}}}'(\hat{y},p)(\phi_y,\phi_p),\phi_u)=(\phi_e,\phi_u).
%
%
    \end{align*}
    Hence, $\VVTforward{\hat{e}}(\hat{e},u) = \id_{T_{\hat{e}}\E\times U}$. 
    By continuity of $(e,u)\mapsto \VVTforward{\hat{e}}(e,u)$ and invertibility of $\VVTforward{\hat{e}}(\hat{e},u) = \id_{T_{\hat{e}}\E\times U}$, invertibility persists in a neighbourhood of $(\hat{e},u)$, so $\VVTforward{\hat{e}}(e,u)$ is locally invertible. Hence, $\VVTforward{\hat{e}}$ is a vector forward transport on $T\E\times U$. 
    Using the relation $L_{\hat{y}}'(y, p, u) = L'(\NE{\VTforward{\hat{y}}}(y,p), u)(\NE{\VTforward{\hat{y}}}'(y,p), \mathrm{id}_U)$ and \eqref{eq:VTTEcodomain} we get for $(\phi_y, \phi_p, \phi_u) \in T_{\hat{y}}\Y \times E_{\hat{y}}\times U$, $e= \NE{\VTforward{\hat{y}}}(y,p) \in \E$ and $\phi_e = \NE{\VTforward{\hat{y}}}'(\hat{y}, p)(\phi_y, \phi_p)\in T_{\hat{e}}\E$
    \begin{align*}
        L_{\hat{y}}'(y, p, u)\WWTforward{\hat{y}}(y,u,p)(\phi_y,\phi_p, \phi_u) &= L_{\hat{y}}'(y, p, u)(\WTforward{\hat{y}}(y)\phi_{y}, \phi_p, \phi_u) \\
        &= L'(\NE{\VTforward{\hat{y}}}(y,p), u)(\NE{\VTforward{\hat{y}}}'(y,p)(\WTforward{\hat{y}}(y)\phi_y, \phi_p), \phi_u)\\
        &= L'(e, u)\VVTforward{\hat{e}}(e,u)(\phi_e, \phi_u).
    \end{align*}
\end{proof}
Using the vector forward transport $\VVTforward{\hat{e}}$ on $T\E \times U$ and the vector forward transport \eqref{eq:VTTYEU} on $T\Y\times E_{\hat{y}}\times U$, and their induced connections on $T^*\E \times U^*$ and $T^*\Y\times E_{\hat{y}}^*\times U^*$, we obtain the following relation for the covariant derivatives of $L_{\hat{y}}'$ and $L'$. 
\begin{proposition}\label{prop:RelationSecondDerivatives} Let $\WTforward{\hat{y}} \in \Gamma(\L(\Y\times T_{\hat{y}}\Y, T\Y))$ be a vector forward transport on $T\Y$ and assume that the connection on $T^*\Y\times E_{\hat{y}}^* \times U^*$ is induced by the corresponding vector transport \eqref{eq:VTTYEU}. If the connection on $T^*\E \times U^*$ is induced by the vector forward transport $\VVTforward{\hat{e}}$, then we obtain for $\hat{e}=  (\hat{y}, p)\in \Y\times E_{\hat{y}}$ and $u\in U$ that
    \[
        \nabla^{T^*\Y\times E_{\hat{y}}^*\times U^*} L_{\hat{y}}'(\hat{y}, p, u)(\delta y, \delta p, \delta u)(\phi_y, \phi_p,\phi_u) = \nabla^{T^*\E\times U^*} L'(\hat{e},u)(\delta e, \delta u)(\phi_e, \phi_u),
    \]
    for all $\phi_e = \NE{\VTforward{\hat{y}}}'(\hat{y}, p)(\phi_y, \phi_p)$, $\delta e = \NE{\VTforward{\hat{y}}}'(\hat{y}, p)(\delta y, \delta p)$ for $(\phi_y, \phi_p,\phi_u), (\delta y, \delta p,\delta u) \in T_{\hat{y}}\Y\times E_{\hat{y}}\times U$.
\end{proposition}
\begin{proof}
    Let $\hat{e} \in \E$ with $\pi(\hat{e})=\hat{y}$. 
    Due to the moving test function technique from Proposition \ref{prop:computationDualConnection} the covariant derivative $\nabla^{T^*\Y\times E_{\hat{y}}\times U} L_{\hat{y}}'(\hat{y}, p, u)$ is given by 
    \begin{align*}
   &\left( \nabla^{T^*\Y \times E_{\hat{y}}^*\times U^*} L_{\hat{y}}'(\hat{y}, p, u)(\delta y, \delta p, u)\right)(\phi_{y}, \phi_{p}, \phi_{u}) \\
   & \quad = \frac{d}{d(y,p,u)}\left(L_{\hat{y}}(y,p,u)\big(\WWTforward{\hat{y}}(y,p,u)(\phi_{y}, \phi_{p}, \phi_{u})\big)\right)\big\vert_{(y,p,u) = (\hat{y}, p, u)}(\delta y, \delta p, \delta u).
\end{align*}
    Using Proposition \ref{prop:VTTE} and the chain rule we obtain 
    \begin{align*}
        &\frac{d}{d(y,p,u)}\left(L_{\hat{y}}(y,p,u)\big(\WWTforward{\hat{y}}(y,p,u)(\phi_{y}, \phi_{p}, \phi_{u})\big)\right)\big\vert_{(y,p,u) = (\hat{y}, p, u)}(\delta y, \delta p, \delta u) \\
        &\quad =\frac{d}{d(e,u)}\left(L'(e, u)(\VVTforward{\hat{e}}(e,u)(\phi_e, \phi_u))\right)\big\vert_{(e,u) = (\hat{e}, u)}(\delta e, \delta u)
    \end{align*}
    for $e=\NE{\VTforward{\hat{y}}}(y,p)$,  $\phi_e = \NE{\VTforward{\hat{y}}}'(\hat{y},p)(\phi_y, \phi_p)$ and $\delta e = \NE{\VTforward{\hat{y}}}'(\hat{y}, p)(\delta y, \delta p)$. According to Proposition~\ref{prop:computationDualConnection}, the second line coincides with the covariant derivative $\nabla^{T^*\E}L'(\hat{e}, u)(\delta e, \delta u)(\phi_e, \phi_u)$ if the connection on $T^*\E$ is defined via the vector transport $\VVTforward{\hat{e}}$.
\end{proof}
Proposition~\ref{prop:RelationSecondDerivatives} will yield the equivalence between the computation of the Newton direction $(\delta e_k, \delta u_k)$ by solving the Newton equation 
\[
    \nabla^{T^*\E \times U^*}L'(e_k, u_k)(\delta e_k, \delta u_k) + L'(e_k, u_k) = 0_{(e_k,u_k)}^*
\]
and the computation of the direction $(\delta y_k, \delta p_k, \delta u_k)$ according to Algorithm~\ref{alg:ImplementableNewton} by solving
\[        
    \nabla^{T^*\Y \times E_{y_k}^*\times U^*} L_{y_k}'(y_k, p_k, u_k)(\delta y_k, \delta p_k, \delta u_k) + L_{y_k}'(y_k, p_k, u_k) = 0_{(y_k,y_k,u_k)}^*.
\]
Furthermore, for the equivalence of the steps, we require a retraction on $\E$. 
\begin{proposition} 
    Let $R^\Y : T\Y \to \Y$ be a retraction on $\Y$.  Let $\VTforward{\hat{y}} \in \Gamma(\L(\Y\times E_{\hat{y}}, \E))$ be a vector forward transport on $\E$. Then, the following mapping defines a retraction on $\E$:
    \begin{equation}\label{eq:retrationE}
        R: T\E \to \E, \quad R_{\hat{e}}(\delta e) := \VTforward{\hat{y}}(R^\Y_{\hat{y}}(\delta y))(p+\delta p)
    \end{equation}
    with $\hat{e} = (\hat{y},p) \in \E$ and $\delta e = \NE{\VTforward{\hat{y}}}'(\hat{y},p)(\delta y, \delta p) \in T_{\hat{e}}\E$ for $(\delta y, \delta p)\in T_{\hat{y}}\Y\times E_{\hat{y}}$.
\end{proposition}
\begin{proof}
    Consider the mapping, defined on the right hand side of \eqref{eq:retrationE}:
    \begin{align*}
    \widetilde{R}_{(\hat{y}, p)} : T_{\hat{y}}\Y\times E_{\hat{y}} &\to \E\\
    (\delta y, \delta p) &\mapsto  \VTforward{\hat{y}}(R^\Y_{\hat{y}}(\delta y))(p+\delta p)=\NE{\VTforward{\hat{y}}}(R^\Y_{\hat{y}}(\delta y),p+\delta p).
    \end{align*}
    By invertibility of $\NE{\VTforward{\hat{y}}}'(\hat{y},p)$ by Lemma~\ref{eq:VTprimeInverse} it holds $\delta e=0_{\hat{e}}$, if and only if $(\delta y,\delta p)=0_{(\hat{y},\hat{y})}$. Hence,
    \[
        R_{\hat{e}}(0_{\hat{e}}) = \widetilde{R}_{(\hat{y}, p)}(0_{\hat y},0_{\hat y})=
        \VTforward{\hat{y}}(R^\Y_{\hat{y}}(0_{\hat{y}}))(p) = \VTforward{\hat{y}}(\hat{y})p = \hat{e}.
    \]
    Then, using that $R_{\hat{y}}^\Y$ is a retraction we get by the chain rule for all $(\phi_y, \phi_p) \in T_{\hat{y}}\Y \times E_{\hat{y}}$ :
    \begin{align*}
        \widetilde{R}_{(\hat{y}, p)}'(0_{\hat{y}}, 0_{\hat{y}})(\phi_y, \phi_p) &= \NE{\VTforward{\hat{y}}}'(\hat{y}, p)((R^{\Y}_{\hat y})'(0_{\hat y})\phi_y, \phi_p)
        = \NE{\VTforward{\hat{y}}}'(\hat{y}, p)(\phi_y, \phi_p).
    \end{align*}
     By definition $R_{\hat e}\circ \NE{\VTforward{\hat{y}}}'(\hat{y},p)=\widetilde{R}_{(\hat{y}, p)}$. Thus application of the chain rule to this equality yields
    \begin{equation}\label{eq:derPseudoRet}
     R_{\hat{e}}'(0_{\hat{e}})\NE{\VTforward{\hat{y}}}'(\hat{y}, p)=\widetilde{R}_{(\hat{y}, p)}'(0_{\hat{y}}, 0_{\hat{y}})=\NE{\VTforward{\hat{y}}}'(\hat{y}, p)
    \end{equation}
    Since $\NE{\VTforward{\hat{y}}}'(\hat{y}, p)$ is invertible by Lemma~\ref{eq:VTprimeInverse}, we obtain $R_{\hat{e}}'(0_{\hat{e}})=\id_{T_{\hat{e}}\E}$.
\end{proof}
A similiar construction can be found in \cite{loayza2026retraction}.
In summary, we get the main result of this section:
\begin{theorem}\label{thm:Newton}
    The iterates generated by Algorithm~\ref{alg:ImplementableNewton} coincide with the iterates generated by Newton's method applied to $L' :\E \to T^*\E$ if the connection on $T^*\E$ is induced by the vector transport \eqref{eq:VTTE} and the retraction on $\E$ is given by \eqref{eq:retrationE}.
\end{theorem}
\begin{proof}
    Consider Newton's method applied to $L' : \E\times U\to T^*\E\times U^*$. In the $k$-th iteration the Newton equation
    \begin{equation}\label{eq:NewtonEquationTE}
        \nabla^{T^*\E \times U}L'(e_k, u_k)(\delta e_k, \delta u_k) + L'(e_k, u_k) = 0_{(e_k,u_k)}^*
    \end{equation}
    is solved to compute a Newton direction $(\delta e_k, \delta u_k)\in T_{e_k}\E \times U$. Then, retraction \eqref{eq:retrationE} and an additive update for the $U$-component is applied, i.e., $(e_{k+1}, u_{k+1}) = (R_{e_k}(\delta e_k), u_k + \delta u_k)$. \\
    Using $e_k = (y_k, p_k) = \NE{\VTforward{y_k}}(y_k, p_k)$, the relation 
    \[
        L_{y_k}'(y_k, p_k, u_k) = L'(\NE{\VTforward{y_k}}(y_k,p_k), u_k)(\NE{\VTforward{y_k}}'(y_k,p_k), \mathrm{id}_U) = L'(e_k, u_k)(\NE{\VTforward{y_k}}'(y_k,p_k), \mathrm{id}_U),
    \]
    Proposition~\ref{prop:RelationSecondDerivatives} and the bijectivity of $\NE{\VTforward{y_k}}'(y_k,p_k)$, solving the Newton equation \eqref{eq:NewtonEquationTE} is equivalent to solving the equation
    \[
        \nabla^{T^*\Y \times E_{y_k}^*\times U^*} L_{y_k}'(y_k, p_k, u_k)(\delta y_k, \delta p_k, \delta u_k) + L_{y_k}'(y_k, p_k, u_k) = 0_{(y_k,y_k,u_k)}^*
    \]
    from Algorithm~\ref{alg:ImplementableNewton}.
    Again, using the bijectivity of $\NE{\VTforward{y_k}}'(y_k,p_k)$ and $\pi(\VTforward{y_k}(R_{y_k}^\Y(\delta y_k))(p_k + \delta p_k)) = R_{y_k}^\Y(\delta y_k)$, the Newton step $(e_{k+1}, u_{k+1}) = (R_{e_k}(\delta e_k), u_k+\delta u_k)$ is equivalent to the iterative step $(y_{k+1}, p_{k+1}, u_{k+1}) = (R_{y_k}^\Y(\delta y_k), \VTforward{y_k}(R_{y_k}^\Y(\delta y_k))(p_k+\delta p_k), u_k + \delta u_k)$ from Algorithm~\ref{alg:ImplementableNewton}.
\end{proof}
To obtain local superlinear convergence of Newton's method applied to $L'$ a suitable smoothness assumption, namely \emph{Newton-differentiability} of $L'$ (see \cite[Def. 4.3]{weigl2024newton}) with respect to a \emph{Newton-derivative} $L''$ is required (cf. \cite[Section 5.3]{weigl2024newton}). Moreover, a suitable boundedness assumption for the operator norm of $\left(\nabla^{T^*\E\times U^*}L'(e,u)\right)^{-1} : T_e\E^*\times U^* \to T_e\E\times U$ is needed. To this end, we require a norm on $T_e\E$ and use the norm of a linear operator $F : X^* \to X$ given by
\[
    \Vert F\Vert_{X^*\to X} := \sup_{\Vert \ell\Vert_{X^*}\leq 1} \Vert F\ell\Vert_X \text{ with } \Vert \ell\Vert_{X^*} = \sup_{x\in X}\vert \ell(x)\vert.
\]
Then, we get the following local convergence result for Algorithm~\ref{alg:ImplementableNewton}.
\begin{corollary}\label{cor:localConvergence}
    Let $(e_\star, u_\star)$ be a zero of $L'$
    and let $L' : \E\times U \to T^*\E\times U^*$ be Newton-differentiable at $(e_\star, u_\star) \in \E\times U$ with respect to a Newton-derivative $L''(e,u) \in L(T_e\E \times U, T_{L'(e,u)}(T^*\E \times U^*))$. Assume that for $(e,u)\in \E\times U$ the operator norm of $\left(\nabla^{T^*\E\times U^*}L'(e,u)\right)^{-1}$ is uniformly bounded, i.e., there exists $\beta < \infty$ such that
    \[
        \Vert \left(\nabla^{T^*\E\times U^*}L'(e,u)\right)^{-1}\Vert_{ T_e\E^*\times U^* \to T_e\E\times U} \leq \beta.
    \]
    Then, the iterates $(y_k, p_k, u_k)$ of Algorithm~\ref{alg:ImplementableNewton} converge locally superlinearly to $(e_\star, u_\star) = ((y_\star, p_\star), u_\star)$.
\end{corollary}
\begin{proof}
    Since the iterates of Algorithm~\ref{alg:ImplementableNewton} coincide with Newton iterates for finding a zero of $L'$ due to Theorem~\ref{thm:Newton}, we obtain the desired result by \cite[Prop. 5.11]{weigl2024newton}.
\end{proof}
In order to implement Algorithm~\ref{alg:ImplementableNewton}, we require a representation of the second covariant derivative of $L_{\hat{y}}$ at the current iterate, i.e., $\nabla^{T^*\Y\times E_{\hat{y}}^*\times U^*} L_{\hat{y}}'(\hat{y},p,u)$. In the following, we therefore derive expressions for the partial derivatives of $L_{\hat{y}}'$ with respect to $y$, $p$ and $u$ and evaluate these at $(\hat{y}, p,u)$. 
\subsection{Computation of second derivatives}\label{sec:ComputationSecondDerivatives}
By symmetry of the linear Newton operator, it suffices to compute its lower triangle and obtain the remaining entries by taking adjoints (or transposes after discretization):
\[ \renewcommand{\arraystretch}{1.4}
    \begin{pmatrix}
        \partial_{yy}L_{\hat{y}}(\hat{y},p,u) & \partial_{yu}L_{\hat{y}}(\hat{y},p,u)^* & \vline &\partial_{yp}L_{\hat{y}}(\hat{y},p,u)^*  \\
        \partial_{uy}L_{\hat{y}}(\hat{y},p,u)  & \partial_{uu}L_{\hat{y}}(\hat{y},p,u) & \vline & \partial_{pu}L_{\hat{y}}(\hat{y},p,u)^* \\
        \hline
        \partial_{py}L_{\hat{y}}(\hat{y},p,u)  & \partial_{pu}L_{\hat{y}}(\hat{y},p,u)  & \vline & \partial_{pp}L_{\hat{y}}(\hat{y},p,u) 
    \end{pmatrix}.
\]
We begin by deriving an expression for the most challenging term, $\partial_{yy}L_{\hat{y}}$. We then compute the derivatives $\partial_{uy}L_{\hat{y}}$ and $\partial_{uu}L_{\hat{y}}$ of the stationarity equation, and finish with the derivatives of the state equation with respect to each of the variables.
\paragraph{Derivative of the adjoint equation.}
Consider the first derivative of $L_{\hat{y}}$ given in \eqref{eq:AdjointEquationAugmented}:
\begin{align*}
    \partial_y L_{\hat{y}}(y,u,p)\phi_y = \partial_y J(y,u)\phi^y + \left(\nabla^{\E^*}(A(\cdot) - B(\cdot)u)(y)\phi_y \right)v(y)\nonumber \\
    + (A(y) - B(y)u)\left(\nabla^\E v(y)\phi_y\right) & \text{ for } \phi_y \in T_y\Y,
\end{align*}
where $v(y) = \VTforward{\hat{y}}(y)p$. Remember that since in the adjoint equation this derivative is evaluated at $(\hat{y},p,u)$ the covariant derivative $\nabla^\E v(\hat{y})$ of $v$ vanishes. However, for the computation of second derivatives with respect to $y$ this additional term has to be taken into account. \\ 
We first observe that $\partial_y J(\cdot,u): \Y\to T^*\Y$ is a mapping into the cotangent bundle $T^*\Y$. Moreover, by the product rule for covariant derivatives we can write
\[
    \left(\nabla^{\E^*}(A(\cdot) - B(\cdot)u)(y)\phi_y \right)v(y)
    + (A(y) - B(y)u)\left(\nabla^\E v(y)\phi_y\right) = ((A(\cdot) - B(\cdot)u)(v))'(y)\phi_y.
\] 
Since the derivative of the dual pairing $(A(\cdot) - B(\cdot)u)(v):\Y\to\R$ is a mapping $((A(\cdot) -B(\cdot)u)(v))':\Y\to T^*\Y$ into the cotangent bundle as well, we have to take the covariant derivative $\nabla^{T^*\Y}$ in order to differentiate $\partial_y L_{\hat{y}}$ with respect to $y$. 
\begin{proposition}\label{prop:Lyy}
    Let $\VTforward{\hat{y}} \in \Gamma(\L(\Y\times E_{\hat{y}}, \E))$ be a vector forward transport on $\E$. Assume that $\E$, and correspondingly $\E^*$, are equipped with connections induced by $\VTforward{y}$. Let $\WTforward{y}\in \Gamma(\L(\Y\times T_y\Y, T\Y))$ be a vector forward transport on $T\Y$ inducing a torsion-free connection on $T\Y$. Then, the second covariant derivative of the locally splitted Lagrangian function    
    \begin{align*}
        L_{\hat{y}} = L \circ (\NE{\VTforward{\hat{y}}}, \id_U) : \Y \times E_{\hat{y}} \times U&\to \R \\
            (y, p, u) &\mapsto J(y,u) + (A(y) - B(y)u)(v(y)),
    \end{align*}
    with $v(y) := \VTforward{\hat{y}}(y)p$ for $p\in E_{\hat{y}}$, at $(\hat{y}, p,u)$ in a direction  $\delta y \in T_{\hat{y}}\Y$ is given by 
    \begin{align*}
        \partial_{yy} L_{\hat{y}}(\hat{y},p,u)\delta y \; \phi_y &= \left(\nabla^{T^*\Y}\partial_y J(\hat{y},u)\delta y\right)\phi_y + \left(\nabla^{\E^*}\left(\nabla^{\E^*}\left(A(\cdot)-B(\cdot)u\right)w(\cdot)\right)(\hat y)\delta y\right)p \\
        &\quad + \left(A(\hat{y})-B(\hat{y})u\right)\left(\nabla^{\E}\left(\nabla^\E v(\cdot)w(\cdot)\right)(\hat{y})\delta y\right),
    \end{align*}
    where $w(y) := \WTforward{\hat{y}}(y)\phi_y$ for $\phi_y \in T_{\hat{y}}\Y$.
    Its representation in trivializations reads
    \begin{equation*}
    \begin{aligned}
    &\partial_{yy}J_{\tau}(\hat{y},u)\delta y \, \phi_{y,\tau} + \partial_yJ_{\tau}(\hat{y},u)\big(\WTprime{\hat{y},\tau}{\rightarrow}(\hat{y})\delta y \; \phi_{y,\tau}\big) \\
    &\quad + \left(A_\tau''(\hat{y}) - B''_{\tau}(\hat{y})u\right)\delta y \, \phi_{y,\tau}\, p_\tau+ \left(A_\tau'(\hat{y}) - B'_{\tau}(\hat{y})u\right)\phi_y(\VTprime{\hat{y},\tau}{\rightarrow}(\hat{y})\delta y \, p_\tau) \\
    &\quad + \left(A_\tau'(\hat{y}) - B'_{\tau}(\hat{y})u\right)\delta y\,\big(\VTprime{\hat{y},\tau}{\rightarrow}(\hat{y})\phi_y \, p_\tau\big) + \left(A_\tau(\hat{y}) - B_\tau(\hat{y})u\right)\big(\VTforward{\hat{y},\tau}''(\hat{y})\delta y \, \phi_{y,\tau} \,p_\tau\big)\\
    &\quad + \left(A_\tau'(\hat{y}) - B'_{\tau}(\hat{y})u\right)\big(\WTprime{{\hat{y},\tau}}{\rightarrow}(\hat{y})\delta y \, \phi_{y,\tau}\big)\, p_\tau  + \left(A_\tau(\hat{y}) - B_\tau(\hat{y})u\right)\big(\VTprime{\hat{y},\tau}{\rightarrow}(\hat{y})\big(\WTprime{{\hat{y},\tau}}{\rightarrow}(\hat{y})\delta y\,\phi_{y,\tau}\big)\,p_\tau\big).
    \end{aligned}
    \end{equation*}
\end{proposition}
\begin{proof}
    Since $\partial_y L_{\hat{y}}(\cdot, u, p) : \Y\to T^*\Y$ is a mapping into a dual vector bundle, namely the cotangent bundle $T^*\Y$, we compute the covariant derivative with respect to $y$ according to the moving test function technique from Proposition~\ref{prop:computationDualConnection}.  Thus, using the vector tranport $\WTforward{\hat{y}}$ on $T\Y$ we extend the test function $\phi_y\in T_{\hat{y}}\Y$, to a vector field $w(y) = \WTforward{\hat{y}}(y)\phi_y$, to obtain:
\begin{equation} \label{eq:adjointEquationTransported} 
    \begin{aligned}
        \partial_y L_{\hat{y}}(y,u,p)\, w(y) &= \partial_y J(y,u)\, w(y)+ \left(\nabla^{\E^*}(A(\cdot) - B(\cdot)u)(y)(w(y)) \right)v(y) \\
        &\quad + (A(y) - B(y)u)\left(\nabla^\E v(y)\,w(y)\right).
    \end{aligned}
\end{equation} 
Fixing $u\in U$ and $p\in E_{\hat y}$, this is a function on $\Y$, whose derivative we will take in direction $\delta y\in T_{\hat y}\Y$, and which is a sum of three dual pairings. To proceed we consider the mappings
    \[
        \partial_y J(\cdot,u) : \Y \to T^*\Y, \quad \left(\nabla^{\E^*}(A(\cdot) - B(\cdot)u)\right)w(\cdot) : \Y\to \E^*, \quad \nabla^{\E}v(\cdot)w(\cdot) : \Y \to \E.
    \]
    Due to Proposition~\ref{prop:computationPrimalConnection}, the covariant derivatives $\nabla^\E v(\hat{y})$ as well as $\nabla^{T\Y}w(\hat{y})$ vanish.
    Using this and applying the product rule for covariant derivatives on the right-hand side of \eqref{eq:adjointEquationTransported} at $y=\hat{y}$ in a direction $\delta y\in T_{\hat{y}}\Y$, we get in terms of covariant derivatives:
\begin{align*}
        \partial_{yy}L_{\hat{y}}(\hat{y},p,u)\delta y \; \phi_y &= \left(\nabla^{T^*\Y}\partial_y J(\hat{y},u)\delta y\right){\phi}_y + \left(\nabla^{\E^*}\left(\nabla^{\E^*}\left(A(\cdot)-B(\cdot)u\right)w(\cdot)\right)(\hat{y})\delta y\right)v(\hat{y}) \\
        &\quad + \left(A(\hat{y})-B(\hat{y})u\right)\left(\nabla^{\E}\left(\nabla^\E v(\cdot)w(\cdot)\right)(\hat{y})\delta y\right).
\end{align*}
Using that $v(\hat{y})=p$, we obtain the first part of the Proposition. 

It remains to compute a representation of this covariant derivative in trivializations. To this end, consider again $\partial_y L_{\hat{y}}$ with the moving test function $w(y)$, i.e., 
\[
    \partial_y L_{\hat{y}}(y,p,u)\big(w(y)\big) = \partial_y J(y,u)\big(w(y)\big) + \left((A(\cdot)-B(\cdot)u)(v)\right)'(y)\big(w(y)\big)
\]
which coincides with \eqref{eq:adjointEquationTransported} by the product rule for covariant derivatives. Differentiating this with respect to $y$ at $y=\hat{y}$ in direction $\delta y\in T_{\hat{y}}\Y$ by using the product rule in trivializations and $w_\tau(\hat{y}) = \phi_{y,\tau}$, we get the representation
\begin{align*}
    \partial_{yy}L_{\hat{y}}(\hat{y},p,u)\delta y \,\phi_y=\;&\partial_{yy}J_{\tau}(\hat{y},u)\delta y\,\phi_{y,\tau} + \partial_y J_\tau(\hat{y},u)\big(w_\tau'(\hat{y})\delta y\big)\\
    &\quad + \left((A(\cdot)-B(\cdot)u)(v)\right)_\tau''(\hat{y})\delta y \, \phi_{y,\tau} + \left((A(\cdot)-B(\cdot)u)(v)\right)_\tau'(\hat{y})\big(w_\tau'(\hat{y})\delta y\big).
\end{align*}
The representation of $\left((A(\cdot)-B(\cdot)u)(v)\right)'(y)\psi$ for $\psi \in T_y\Y$ is according to \eqref{eq:linStateEqTrivialization} given by 
\begin{equation}\label{eq:DualPairingTrivialization}
    \left((A(\cdot)-B(\cdot)u)(v)\right)_\tau'(y)\psi = 
    (A_\tau'(y) - B'_{\tau}(\hat{y})u)\psi\,v_\tau(\hat{y}) + (A_\tau(y)-B_\tau(y)u)v_\tau'(y)\psi.
\end{equation}
Differentiating this with respect to $y$ at $y=\hat{y}$ in a direction $\delta y\in T_{\hat{y}}\Y$ and testing with $\phi_y$ yields the representation
\begin{equation}\label{eq:secondDerivABTrivialization}
\begin{aligned}
    &\left((A(\cdot)-B(\cdot)u)(v)\right)_\tau''(\hat{y})\delta y \, \phi_{y,\tau}\\
    &\quad=(A_\tau''(\hat{y}) - B''_{\tau}(\hat{y})u)\delta y \, \phi_{y,\tau}\,v_\tau(\hat{y}) + (A_\tau'(\hat{y}) - B'_{\tau}(\hat{y})u)\phi_y\,v'_\tau(\hat{y})\delta y\\
    &\quad\quad + (A_\tau'(\hat{y})-B'_{\tau}(y)u)\delta y \, v_\tau'(\hat{y})\phi_{y,\tau} + (A_\tau(\hat{y})-B_\tau(\hat{y})u)v_\tau''(\hat{y})\delta y \,\phi_{y,\tau}.
\end{aligned}
\end{equation}
Since $v(y) = \VTforward{\hat{y}}(y)p$ we get $v_\tau(\hat{y}) = p_\tau$, and the derivatives of $v$ in trivializations read $v_\tau'(\hat{y})\delta y = \VTprime{\hat{y},\tau}{\rightarrow}(\hat{y})\delta y \, p_\tau$ and $v_\tau''(\hat{y})\delta y \, \phi_{y,\tau} = \VTforward{\hat{y},\tau}''(\hat{y})\delta y \, \phi_{y,\tau} \, p_\tau$. Plugging these expressions into \eqref{eq:secondDerivABTrivialization} and setting $\psi = w_\tau'(\hat{y})\delta y = \WTprime{\hat{y},\tau}{\rightarrow}(\hat{y})\delta y \, \phi_{y,\tau}$ in \eqref{eq:DualPairingTrivialization}, we obtain the desired representation.
\end{proof}
    The second part of the proof of Proposition~\ref{prop:Lyy} provides a practical and constructive way to compute the second covariant derivative of a Lagrangian function by using the moving test functions technique from Proposition~\ref{prop:computationDualConnection}, twice. First, we define the locally splitted Lagrangian $L_{\hat{y}}$ and extend $p$ to a section $v(y) := \VTforward{\hat{y}}(y)p$ on $\E$, and compute the first derivative of the map $y \mapsto L_{\hat{y}}(y,u,v(y))$ with respect to $y$ in a direction $\phi_y\in T_{\hat y}\Y$. Then we extend the test function $\phi_y$ via a vector forward transport to a section $w(y):=\WTforward{\hat{y}}(y)\phi_y$ on $T\Y$ and differentiate the map $y \mapsto \partial_yL_{\hat{y}}(y,u,v(y))w(y)$ in directon $\delta y\in T_{\hat y}\Y$. For both derivatives we use the product rule in trivializations. 
    
    Since the product rule can also be applied in embeddings, the same technique can be used in embedded manifolds and yields the same terms. We used this alternative to computations in trivializations in our numerical examples, below.\\
    \;
    For the implementation we then only require the first and second derivatives of $A$ and $B$ and the first and second derivatives of $\VTforward{\hat y}$, as well as the first derivative of $\WTforward{\hat y}$. In many applications, these ingredients are explicitly available or can be computed with reasonable effort.
    Consequently, despite its unwieldy appearance, the formula for the second covariant derivative is readily implementable in practice.

\paragraph{Derivatives of the stationarity equation.}
Consider the partial derivative $\partial_u L_{\hat{y}}$ given in \eqref{eq:stationarityEquation}:
\begin{equation*}
    \partial_u L_{\hat{y}}(y,u,p)\phi_u = \partial_u J(y,u) \phi_u - (B(y)\phi_u)v(y) \text{ for } \phi_u \in U.
\end{equation*}
Evaluating this at $y=\hat{y}$ yields the stationarity equation. Since the second derivatives of $L_{\hat{y}}$ are symmetric, we use $\partial_{up}L_{\hat{y}}(\hat{y},u,p) = \partial_{pu}L_{\hat{y}}(\hat{y},u,p)^*$, and thus we only have to compute the derivatives $\partial_{uy}L_{\hat{y}}$ and $\partial_{uu}L_{\hat{y}}$.

For the derivative with respect to $y$ at $y=\hat{y}$, we have to take into account that $B(\cdot)\phi_u : \Y\to \E^*$ for a fixed $\phi_u \in U$, as well as $v : \Y\to\E$, map into vector bundles. Thus, by using the product rule for covariant derivatives we obtain for $\phi_u\in U$ and $\delta y \in T_{\hat{y}}\Y$:
\begin{align*}
    \partial_{uy}L_{\hat{y}}(\hat{y},p,u)\delta y\, \phi_u = \partial_{uy}J(\hat{y},u)\delta y \; \phi_u - \left(\nabla^{\E^*}(B(\cdot)\phi_u)(\hat{y}) \delta y\right) v(\hat{y}) - B(\hat{y})\phi_u\left(\nabla^\E v(\hat{y}) \delta y\right).
\end{align*}
Again, using that the covariant derivative of $v(y) = \VTforward{\hat{y}}(y)p$ with respect to $y\in \Y$ vanishes at $y=\hat{y}$ due to Proposition~\ref{prop:computationPrimalConnection}, and $v(\hat{y}) = p$ this reduces to 
\[
    \partial_{uy}L_{\hat{y}}(\hat{y},p, u)\delta y\, \phi_u = \partial_{uy}L(\hat{y},u)\delta y \; \phi_u - \nabla^{\E^*}(B(\cdot)\phi_u)(\hat{y}) \delta y \; p \quad \text{ for } \phi_u \in U \text{ and } \delta y\in T_{\hat{y}}\Y.
\]
In trivializations, this reads according to \eqref{eq:Qstartriveq}
\begin{equation} \label{eq:Luy}
    \partial_{uy}L_{\hat{y}}(\hat{y},p, u)\delta y \,\phi_u = \partial_{uy}J_\tau(\hat{y},u)\delta y \; \phi_{u,\tau} - (B'_\tau(\hat{y})\phi_u\delta y)\,p_\tau - B_\tau(\hat{y})\phi_u\left(\VTprime{\hat{y},\tau}{\rightarrow}(\hat{y})\delta y \, p_\tau\right).
\end{equation}
The derivative with respect to $u$ is given by the ordinary second derivative of $J$ since $U$ is a linear space, i.e., we obtain
\[
    \partial_{uu}L_{\hat{y}}(\hat{y},p, u) \delta u \, \phi_u= \partial_{uu} J(\hat{y},u)\delta u \; \phi_u \quad \text{ for } \phi_u, \delta u \in U.
\]
\paragraph{Derivatives of the state equation.}
Consider the partial derivative
\begin{equation*}
    \partial_p L_{\hat{y}}(y,u,p)\phi_p = (A(y) - B(y)u)(\VTforward{\hat{y}}(y)\phi_p) \text{ for } \phi_p \in E_{\hat{y}}.
\end{equation*}
The derivative with respect to $y\in\Y$ yields the linearized state equation \eqref{eq:linearizedStateEq} which has a representation \eqref{eq:linStateEqTrivialization}. Since the covariant derivative of $\VTforward{\hat{y}}(y)\phi_p$ at $y=\hat{y}$ is zero, this reduces to
\[
    \partial_{py}L_{\hat{y}}(\hat{y},p,u) \delta y\,\phi_p = \left(\nabla^{\E^*}(A(\cdot) - B(\cdot)u)(\hat{y})\delta y \right)\phi_p \quad \text{ for } \phi_p \in E_{\hat{y}} \text{ and } \delta y\in T_{\hat{y}}\Y.
\]
For the derivative with respect to $u \in U$ we use that $B(y)$ is linear in $u$ for fixed $y\in \Y$. Thus, we get
\[
    \partial_{pu}L_{\hat{y}}(\hat{y},p,u)\delta u \, \phi_p = - (B(\hat{y})\delta u)\phi_p \quad \text{ for } \phi_p\in E_{\hat{y}} \text{ and }\delta u \in U.
\]
Obviously, the derivative of $\partial_p L_{\hat{y}}$ with respect to $p$ is zero. \\
We have now computed expressions for every ingredient required for the implementation of Algorithm~\ref{alg:ImplementableNewton}. The choice of the vector transports and the retraction, of course depends on the specific application. We will provide some examples in the next section.
\section{Numerical examples}\label{sec:Numerics}
In this section we will illustrate our results by a couple of examples. First we consider the optimal control of geodesics and surfaces of minimal energy by force fields. Then, we consider two slightly more involved applications from continuum mechanics, namely inextensible rods and shells. All our example arise from energy minimization principles, and thus we have $\E=T\Y$ in all cases. 
We deliberately do not touch questions of (numerical) analysis, since this is beyond the scope of the current work. 

We will present our problems to be solved in the continuous setting. However, the implementation requires discretization by finite elements, which we will only briefly describe for completeness. 
In our examples $\Y$ is a manifold of mappings into a manifold $\M$ that has a natural embedding into a Euclidean vector space. 
For discretization, we can replace $\Y$ by a finite-dimensional product manifold $\Y_h$ and correspondingly $\E$ by $\E_h$ and $U$ by $U_h$. Then, we choose our ansatz functions $x:=(y,p,u)\in \E_h\times U_h$ with matching values at the interpolation nodes $\omega_i$, i.e. $y(\omega_i) \in \M$ and $p(\omega_i) \in T_{y(\omega_i)}\M$, and via polynomial interpolation in the embedding space.
 The test functions $\phi_x := (\phi_y,\phi_p,\phi_u)\in T_y\Y_h\times E_{y,h}\times U_h$ and the search directions $\delta x:=(\delta y,\delta p,\delta u)\in T_y\Y_h\times E_{y,h}\times U_h$ are chosen from finite element spaces, which map into the corresponding tangent spaces at the interpolation nodes. 
To this end, we map standard Lagrange finite basis functions, provided by the library \texttt{Ferrite.jl}~\cite{carlsson_ferrite_2025}, to orthonormal bases of tangent spaces $T_{y_i}\M$ available in \texttt{Manifolds.jl}\cite{AxenBaranBergmannRzecki:2023} at each interpolation node $y_i = y(\omega_i)$. The required retractions are also provided by \texttt{Manifolds.jl}.

Algorithm~\ref{alg:ImplementableNewton} is implemented in \texttt{Julia}~\cite{BezansonEdelanKarpinskiViral:2017} using Newton's method from \texttt{Manopt.jl}~\cite{Bergmann2022}\footnote{see the documentation at \url{https://manoptjl.org/stable/solvers/vectorbundle_newton/}}. The vectors and matrices, representing the first and second derivatives of $L_{y}$, required for the computation of the Newton steps, are assembled in a standard fashion by  evaluating the integral expressions, derived in this paper for (pairs of) finite element basis functions, using standard quadrature. All linear systems were solved by the sparse linear solver, which is built into \texttt{Julia}. 
\subsection{Optimal control of geodesics and surfaces of minimal energy}
As a first application we consider the optimal control of elastic geodesics and surfaces of minimal energy. Let $\Omega \subset \R^d$, $d=1,2$, be a bounded domain. We consider a mapping $y : \Omega \to \M$ into an isometrically embedded submanifold $\M \subset \R^3$ which minimizes the Dirichlet energy 
\[
    E(y) := \frac{1}{2} \int_\Omega \langle y'(\omega), y'(\omega) \rangle_{F} \, d\omega.
\]
By additionally requiring boundary conditions, we obtain a geodesic in the case $d=1$, and a surface of minimal energy confined to $\M$ for $d=2$.\\
\;\\
We consider $U := L_2(\Omega, \R^3)$ and the Sobolev manifold (cf. e.g. \cite{hajlasz2009sobolev})
\begin{equation*}
    \Y := H^1(\Omega, \M) \coloneqq \{ g \in H^1(\Omega, \R^3) \; \mid \; g(\omega) \in \M \text{ almost everywhere} \}. 
\end{equation*}
Then, as discussed in Section~\ref{sec:VariationalProblemsAndOCP}, by introducing a control $u \in U$ acting on the state $y\in \Y$, we obtain an elastic geodesic, or an elastic surface of minimal energy, creating a balance of forces by solving the state equation
\[
     (A(y)-B(y)u)\phi = \int_\Omega \langle y'(\omega), \phi'(\omega)\rangle_{F} + \langle u(\omega), \phi(\omega) \rangle_{2} \, d\omega = 0 \quad \text{ for all } \phi \in E_y.
\]
Here, we set $\E := T\Y$ and define
\begin{align*}
    A : \Y &\to \E^* & B : \Y\times U &\to \E^*\\
    A(y)\phi &:= \int_\Omega \langle y'(\omega), \phi'(\omega)\rangle_{F} \, d\omega, & (B(y)u)\phi &:= - \int_\Omega \langle u(\omega), \phi(\omega)\rangle_{2} \, d\omega.
\end{align*}
We use a standard tracking-type functional to penalize the deviation of the state $y$ from a desired configuration $y_d \in \Y$. Here, we substract $y_d$ from $y$ in the embedding space. Using a regularization parameter $\alpha >0$ we then consider the geometric optimal control problem 
\begin{align*}
    &\min_{(y,u)\in \Y\times U} \frac{1}{2}\|y-y_d\|_{L^2(\Omega)}^2 + \frac{\alpha}{2}\Vert u \Vert_{L^2(\Omega)}^2 \\
    & \hspace{0.5cm} s.t.  \int_\Omega \langle y'(\omega), \phi'(\omega)\rangle_{F} + \langle u(\omega), \phi(\omega)\rangle_{2} \, d\omega = 0 \quad \text{ for all } \phi \in E_y.
\end{align*} 
For the treatment of the optimal control problem, we would like to define the locally splitted Lagrangian function. To this end, we require a vector transport on $\E = T\Y$. First, using the embedding of $T\M$ into $\R^n$ given by $\iota_{\R^3}(\eta)\in L(T_{\eta}\M, \R^3)$ for $\eta\in \M$, we can define a vector forward transport on $T\M$ by using orthogonal projections $P(\xi)\colon \R^3 \to T_{\xi}\M$ for $\xi\in \M$:
\begin{equation*}
    \VTforward{\eta}(\xi) := P(\xi)\iota_{\R^3}(\eta) \in L(T_{\eta}\M, T_{\xi}\M).
\end{equation*}
With this, we can construct a vector forward transport $\VTforward{\hat{y}}(y) : T_{\hat{y}}\Y \to T_y\Y$ on $\E = T\Y$ in a pointwise fashion by setting 
\begin{equation}\label{eq:VTH1}
    (\VTforward{\hat{y}}(y)\phi) (\omega) :=(\VTforward{\hat{y}(\omega)}(y(\omega)))\phi(\omega) = P(y(\omega))\iota_{\R^3}(\hat{y}(\omega))\phi(\omega) \text{ for each } \omega\in \Omega,
\end{equation}
where $\phi  \in T_{\hat{y}}\Y$, i.e., $\phi(\omega) \in T_{\hat{y}(\omega)}\M$. For $\hat{y}\in \Y$, this yields the locally splitted Lagrangian
\begin{align*}
    L_{\hat{y}}:\Y\times E_{\hat{y}}\times U &\to \R \\
    (y,p,u) &\mapsto \frac{1}{2}\|y-y_d\|_{L^2(\Omega)}^2 + \frac{\alpha}{2}\Vert u \Vert_{L^2(y)}^2 \\
    &\qquad + \int_\Omega \langle y'(\omega), \left(\VTforward{\hat{y}(\cdot)}(y(\cdot))p(\cdot)\right)'(\omega)\rangle_{F} + \langle u(\omega), \VTforward{\hat{y}(\omega)}(y(\omega))p(\omega)\rangle_{2} \, d\omega.
\end{align*}
Differentiating $L_{\hat{y}}$ with respect to $y$, $u$ and $p$ at $(y,u,p) = (\hat{y}, p,u)$ yields the optimality conditions corresponding to the optimal control problem. The representation of the \eqref{eq:adjointEq} is given by \eqref{eq:adjointEquationTrivialization} where we compute the derivative of the vector transport in the embedding space $\R^3$, i.e., we substitute $(\VTprime{\hat{y}, \tau}{\rightarrow}(\hat{y})\phi_y \,p)(\omega) = P'(\hat{y}(\omega))\phi_y(\omega)p(\omega)$ for $\phi_y \in T_{\hat{y}}\Y$ and $\omega \in \Omega$. For the vector transport $\WTforward{\hat{y}}$ on $T\Y$ required for the computation of the second covariant derivative $\partial_{yy}L_{\hat{y}}(\hat{y}, p,u)$ according to Proposition~\ref{prop:Lyy}, we choose \eqref{eq:VTH1} as well. Then, the Newton matrix, can be determined due to the representations given in Section~\ref{sec:ComputationSecondDerivatives}. For the sake of brevity, we omit the explicit calculation of the usual first and second derivatives of $A$ and $B$. The representation of the second derivative of $\VTforward{\hat{y}}$ with respect to $y$ at $y=\hat{y}$ is given by $(\VTforward{\hat{y}, \tau}''(\hat{y})\delta y \,\phi_y \,p)(\omega) = P''(\hat{y}(\omega))\delta y(\omega)\phi_y(\omega)p(\omega)$ for $\delta y, \phi_y \in T_{\hat{y}}\Y$ and $\omega \in \Omega$. For the implementation of the Newton step, we again use the vector forward transport $\VTforward{\hat{y}}$. 
As a retraction $R^\Y$ on $\Y = H^1(\Omega, \M)$ we use a retraction $R^\M$ on $\M$ in a pointwise fashion, i.e. for $y \in \Y$, $\delta y \in T_{y}\Y$ and for each $\omega \in \Omega$ we set
\[
    \left(R_{y}^\Y(\delta y)\right)(\omega) := R^\M_{y(\omega)}(\delta y(\omega)).
\]
In our numerical examples we choose the unit sphere $\M = \mathbb{S}^2 \subset \R^3$ equipped with the Riemannian metric given by the Euclidean inner product restricted to each embedded tangent space. Then, the orthogonal projection onto the tangent space $T_\xi \mathbb{S}^2 \cong \xi^\perp$ is given by
\begin{equation*}
    P(\xi)\colon \R^3 \to T_\xi\mathbb{S}^2, \quad P(\xi) := \id - \xi\langle \xi, \cdot \rangle_2.
\end{equation*}
The derivative of the orthogonal projection at $\hat{\xi}\in\mathbb{S}^2$ in the direction $\phi_\xi \in T_{\hat{\xi}}\mathbb{S}^2$ applied to $v\in T_{\hat{\xi}}\mathbb{S}^2$ is given by
\begin{equation*}
    (P'(\hat{\xi})\phi_\xi)v = -\hat{\xi}\langle \phi_\xi, v\rangle_2 - \phi_\xi \langle \hat{\xi}, v \rangle_2.
\end{equation*}
Its second derivative in the direction $\delta \xi \in T_{\hat{\xi}}\mathbb{S}^2$ is then given by 
\[
    (P''(\hat{\xi})\delta \xi \, \phi_\xi)v = -\delta \xi\langle \phi_\xi, v\rangle_2 - \phi_\xi \langle \delta \xi, v \rangle_2.
\]
As a retraction $R^\M$ on the sphere we use the retraction given by the normalization, i.e., for $\xi \in \mathbb{S}^2$, $\delta \xi \in T_{\xi}\mathbb{S}^2$ we set
\begin{equation*}
    R^\M_{\xi}(\delta \xi) := \frac{\xi+\delta \xi}{\lVert\xi+\delta\xi\rVert_2} \in \mathbb{S}^2.
\end{equation*}
 Figure~\ref{fig:geodesicForce} shows the local superlinear convergence that we expected to see due to Corollary \ref{cor:localConvergence}. In addition, the resulting controlled geodesic can be seen, as well as the initial guess and the control acting on each point of the curve. Figure~\ref{fig:minimalSurface} shows a similar result for the optimal control of a surface of minimal energy. 
\begin{figure}[tbp]
\begin{minipage}[t]{0.54\textwidth}
        \begin{tikzpicture}
        \begin{axis}[
            xlabel={Iteration},
            ylabel={$\Vert (\delta y, \delta p, \delta u)\Vert$},
            ymode=log,
            grid=both,
            width=12cm,
            height=8.5cm,
            scale=0.4,
            xtick={1,2,3,4,5},
            ytick distance=10^4
        ]
                \addplot table [
            col sep=comma,
            header=false,
            x expr=\coordindex+1,   
            y index=0             
            ]{data/norm_newton_direction_geodesic_OCP.csv};
        \end{axis}
        \end{tikzpicture}
\end{minipage}
\begin{minipage}[t]{0.45\textwidth}
    \centering
    \includegraphics[scale=0.5]{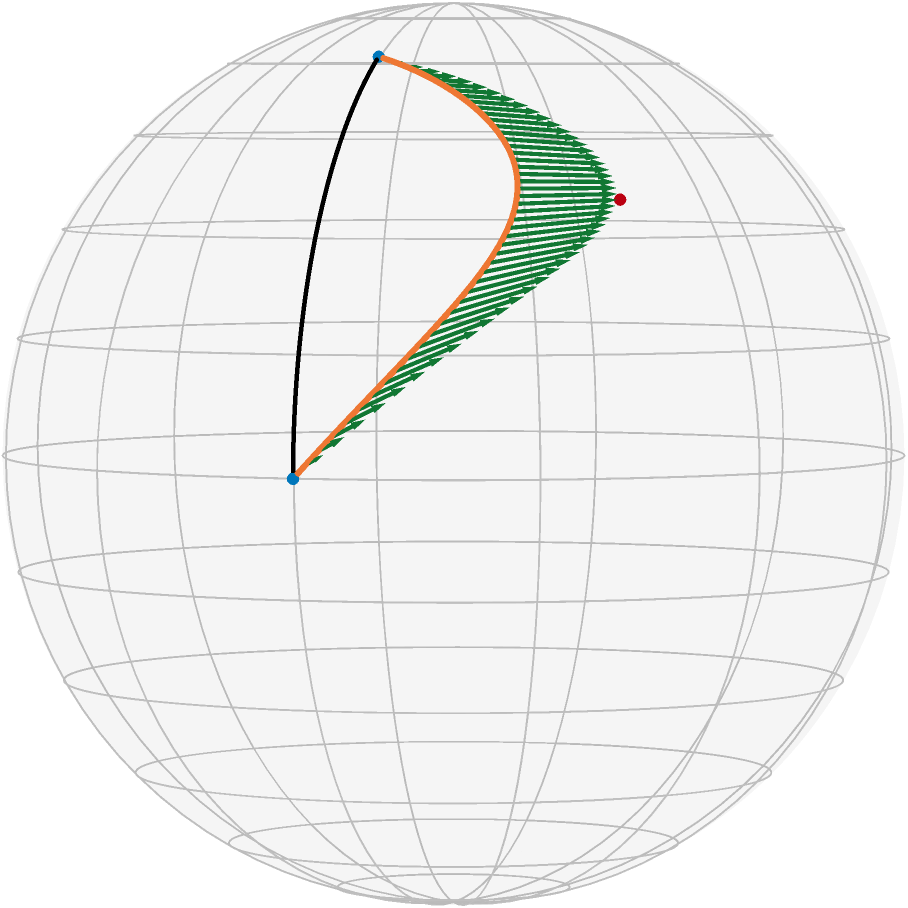}
\end{minipage}
\caption{Superlinear convergence of Algorithm~\ref{alg:ImplementableNewton} for $50$ discretization points and $\alpha = 0.015$ (left). Initial geodesic (black) with Dirichlet boundary conditions (blue), desired state $y_d \equiv \frac{1}{\sqrt{3}}\cdot(1,1,1)^T$ (red) and resulting elastic geodesic (orange). The control is illustrated by green arrows (right).}
\label{fig:geodesicForce}
\end{figure}
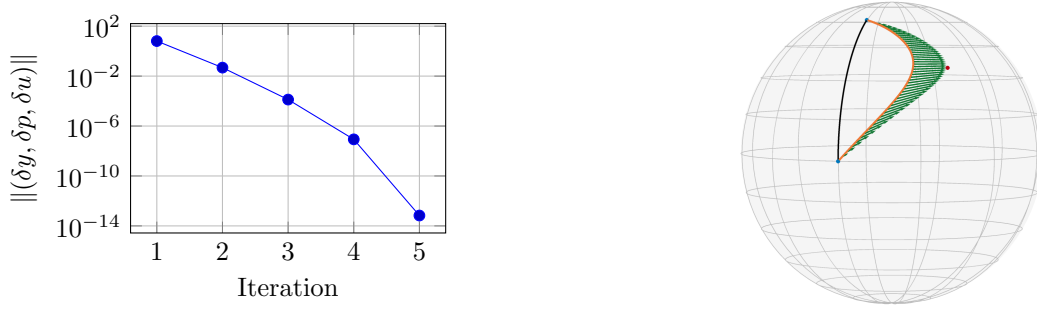

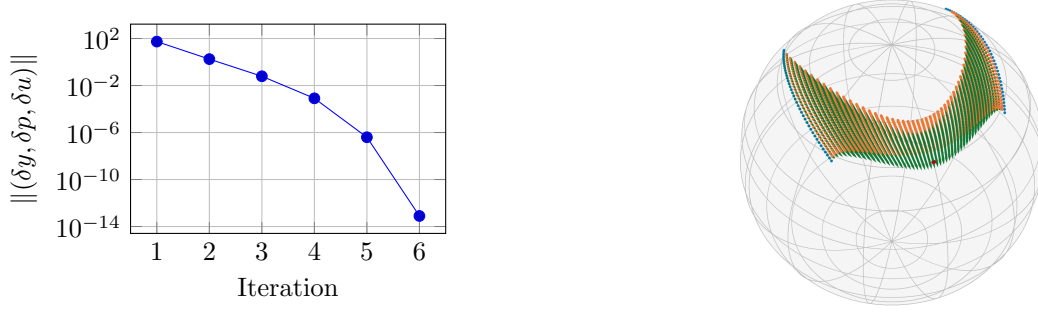
\begin{figure}[tb]
\begin{minipage}[t]{0.54\textwidth}
        \begin{tikzpicture}
        \begin{axis}[
            xlabel={Iteration},
            ylabel={$\Vert (\delta y, \delta p, \delta u)\Vert$},
            ymode=log,
            grid=both,
            width=12cm,
            height=8.5cm,
            scale=0.4,
            xtick={1,2,3,4,5,6},
            ytick distance=10^4
        ]
                \addplot table [
            col sep=comma,
            header=false,
            x expr=\coordindex+1,   
            y index=0             
            ]{data/norm_newton_direction_minimalsurface_OCP.csv};
        \end{axis}
        \end{tikzpicture}
\end{minipage}
\begin{minipage}[t]{0.45\textwidth}
    \centering
    \includegraphics[scale=0.5]{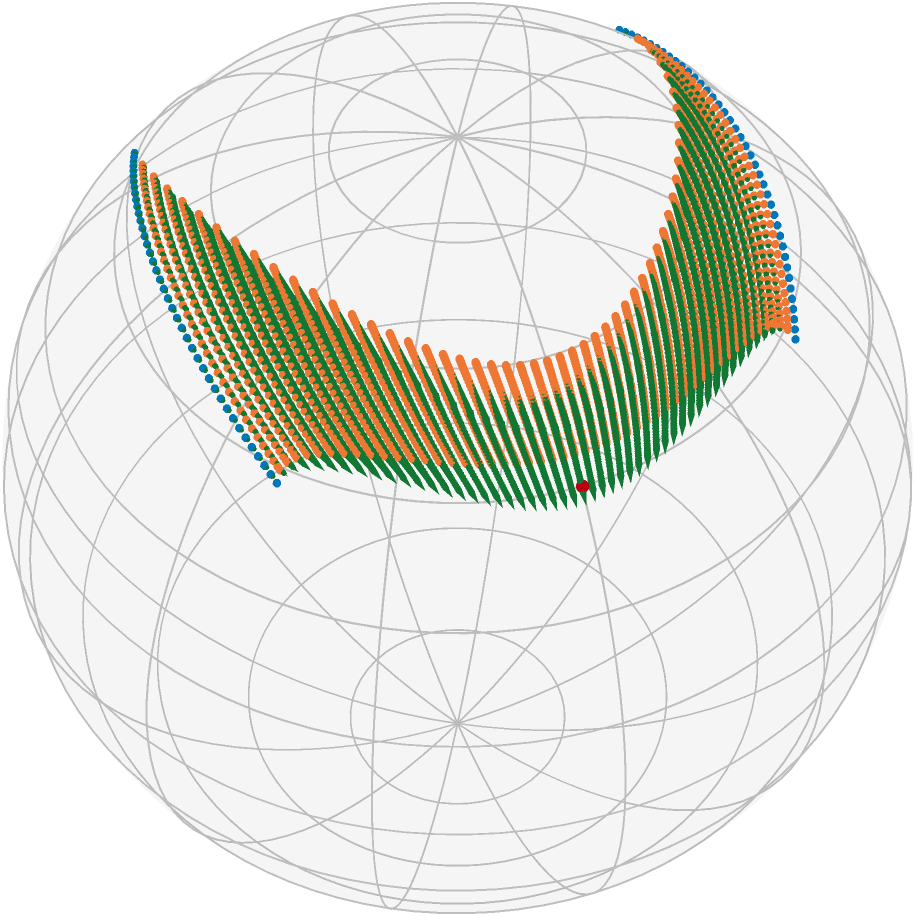}
\end{minipage}
\caption{Superlinear convergence of Algorithm~\ref{alg:ImplementableNewton} for $\alpha = 0.05$ (left). Desired state $y_d \equiv \frac{1}{\sqrt{2}}\cdot(0,1,1)^T$ (red) and resulting surface of minimal energy (orange) with Dirichlet boundary conditions (blue). The control is illustrated by green arrows (right).}
\label{fig:minimalSurface}
\end{figure}

\subsection{Optimal control of inextensible rods and shells}

As a second application we consider a class of problems from continuum mechanics, namely inextensible rods and shell.  Consider a domain $\Omega \subset \R^d$, $d=1,2$. We are looking for a deformation mapping $w : \Omega \to \R^3$. In the case $d=1$ the state $w$ models the deformation of a flexible rod, while $w$ models the deformation of a flexible shell in the case $d=2$. We do not consider the case $d=3$, since an inextensibility constraint in this case would yield simply the rigid body motions.  
For simplicity of presentation, we use the domain $\Omega$ both as the parameter domain, as well as the undeformed body in the absence of forces, so that $w(\Omega)$ is the shape of the deformed body. 

For the spatial derivative of the deformation we introduce the new variable with the aim to derive a mixed formulation of the problem:
\[
 v : \Omega \to \R^{3\times d} : v(\omega) = w'(\omega).
\]
To model \emph{inextensible} bodies we require that $v(\omega)\in \R^{3\times d}$ is an orthogonal matrix, i.e., an element of the compact Stiefel manifold $St(3,d)$ (which is the unit sphere $\mathbb S^2$ in the case $d=1$): 
\[
 v(\omega) \in St(3,d) = \{ Q \in \R^{3\times d} : Q^TQ=I_d \}.
\]
As a control we consider a force field $u \in L_2(\Omega,\R^3)$, attached to the body so that the \emph{forward problem} can be written as the following energy minimization problem:
\[
 \min_{(w,v)} \int_\Omega \frac12\langle v'(\omega),v'(\omega)\rangle_{F,3}-\langle u(\omega),w(\omega)\rangle_{2} \, d\omega \quad \text{s.t.} \quad v(\omega) -w'(\omega)=0.
\]
Here, $\langle\cdot, \cdot\rangle_{F,3}$ is defined for tensors of order $k=3$ in $\R^{3\times d\times d}$ by the generalized Frobenius inner product:
\[
 \langle A,B\rangle_{F,3} := \sum_{i=1}^3\sum_{j_1, j_{2}=1}^d A_{ij_1 j_{2}}B_{ij_1 j_2}.
\]
The first quadratic term in the energy is a bending energy in terms of $v'(\omega) \in \R^{3 \times d \times d}$.
Computation of the optimality conditions of this problem for a fixed force field $u$ yields the following variational problem with a Lagrangian multiplier $\lambda \in L_2(\Omega,\R^{3\times d})$:
\begin{align*}
 0 &= \int_\Omega \langle v'(\omega),\phi_v'(\omega)\rangle_{F,3}+\langle \lambda(\omega),\phi_v(\omega)\rangle_{F} \, d\omega \quad \forall \phi_v \in T_v H^1(\Omega,St(3,d))\\
 0 &= \int_\Omega -\langle u(\omega),\phi_w(\omega)\rangle_{2}-\langle\lambda(\omega),\phi'_w(\omega)\rangle_{F} \, d\omega\quad \forall \phi_w \in H^1(\Omega,\R^3)\\
 0 &= \int_\Omega  \langle v(\omega) -w'(\omega),\phi_\lambda(\omega)\rangle_{F}\, d\omega \hspace{2cm} 
 \forall \phi_\lambda \in L_2(\Omega,\R^{3\times d}).
\end{align*}
This system substantiates the state equation of our optimal control problem. The state is defined as a triplet on a product manifold
\[
y=(w,v,\lambda) \in \mathcal Y :=  H^1(\Omega,\R^3)\times H^1(\Omega,St(3,d)) \times L_2(\Omega,\R^{3\times d}),
\]
while the control is simply the applied force field $u\in L_2(\Omega,\R^3)$. With the combined test function 
\[
 \phi = (\phi_w,\phi_v,\phi_\lambda) \in H^1(\Omega,\R^3)\times T_vH^1(\Omega,St(3,d)) \times L_2(\Omega,\R^{3\times d}) = T_y \mathcal Y^*
\]
we can write
\begin{align*}
  A : \mathcal Y &\to T\mathcal Y^*\\
  A(y)\phi &= \int_\Omega \langle v'(\omega),\phi_v'(\omega)\rangle_{F,3}+\langle \lambda(\omega),\phi_v(\omega)-\phi'_w(\omega)\rangle_{F}+\langle v(\omega) -w'(\omega),\phi_\lambda(\omega)\rangle_{F}\, d\omega
\end{align*}
and 
\begin{align*}
  B : L_2(\Omega,\R^{3}) &\to T\mathcal Y^*\\
  Bu\, \phi &=\int_\Omega \langle u(\omega),\phi_w(\omega)\rangle_{2} \, d\omega.
\end{align*}  
In contrast to the previous example, there is no implicit dependence on $y$ in $B$, since the test functions used for the definition of $B$ are independent of $y$. 

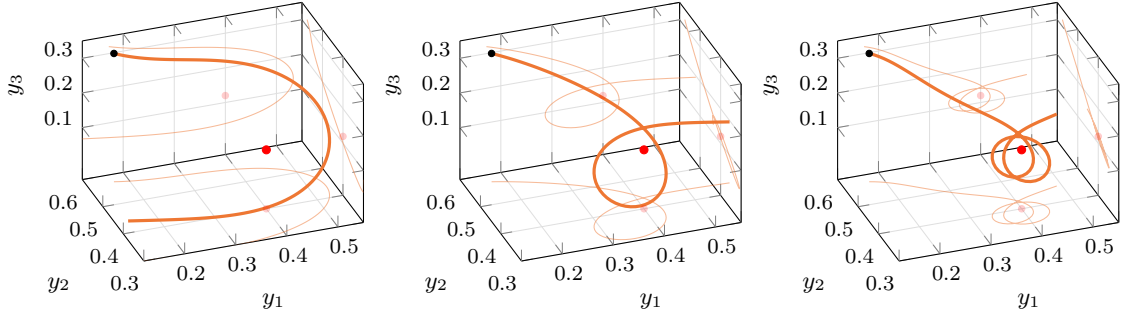
\begin{figure}[t]
\begin{minipage}[t]{0.32\textwidth}
    \begin{tikzpicture}
\begin{axis}[
    view={-20}{32},
    width=5.3cm, height=5cm,
    plot box ratio=1.3 1 1,
    xmin=\xmin, xmax=\xmax,
    ymin=\ymin, ymax=\ymax,
    zmin=\zmin, zmax=\zmax,
    xtick={0.2,0.3,0.4,0.5}, ytick={0.3,0.4,0.5,0.6}, ztick={0.1,0.2,0.3},
    scaled ticks=false,
    tick label style={font=\footnotesize},
    label style={font=\small},
    xlabel={$y_1$}, ylabel={$y_2$}, zlabel={$y_3$},
    grid=major, grid style={gray!25, line width=0.3pt},
]
 
\pgfplotsset{proj/.style={thin, opacity=0.5, TolVibrantOrange}}
\addplot3[proj] table[x expr=\xmax, y=y1, z=z1, col sep=comma] {data/inextensible-rod-OCP-result-alpha-5.csv};
\addplot3[proj] table[x=x1, y expr=\ymax, z=z1, col sep=comma] {data/inextensible-rod-OCP-result-alpha-5.csv};
\addplot3[proj] table[x=x1, y=y1, z expr=\zmin, col sep=comma] {data/inextensible-rod-OCP-result-alpha-5.csv};
 
\addplot3[very thick, color=TolVibrantOrange]
    table[x=x1, y=y1, z=z1, col sep=comma] {data/inextensible-rod-OCP-result-alpha-5.csv};
 
\addplot3[only marks, mark=*, mark size=1.5pt, red]
    coordinates {(0.4,0.4,0.12)};
\addplot3[only marks, mark=*, mark size=1.2pt, red, opacity=0.2]
    coordinates {(\xmax,0.4,0.12) (0.4,\ymax,0.12) (0.4,0.4,\zmin)};

\addplot3[only marks, mark=*, mark size=1.2pt, black]
     coordinates {(0.17181712472104044,0.5748729213074786,0.32009219983223997)};

\end{axis}
\end{tikzpicture}
\end{minipage}
\begin{minipage}[t]{0.32\textwidth}
    \begin{tikzpicture}
\begin{axis}[
    view={-20}{32},
    width=5.3cm, height=5cm,
    plot box ratio=1.3 1 1,
    xmin=\xmin, xmax=\xmax,
    ymin=\ymin, ymax=\ymax,
    zmin=\zmin, zmax=\zmax,
    xtick={0.2,0.3,0.4,0.5}, ytick={0.3,0.4,0.5,0.6}, ztick={0.1,0.2,0.3},
    scaled ticks=false,
    tick label style={font=\footnotesize},
    label style={font=\small},
    xlabel={$y_1$}, ylabel={$y_2$}, zlabel={$y_3$},
    grid=major, grid style={gray!25, line width=0.3pt},
]
 
\pgfplotsset{proj/.style={thin, opacity=0.5, TolVibrantOrange}}
\addplot3[proj] table[x expr=\xmax, y=y1, z=z1, col sep=comma] {data/inextensible-rod-OCP-result-alpha-7.csv};
\addplot3[proj] table[x=x1, y expr=\ymax, z=z1, col sep=comma] {data/inextensible-rod-OCP-result-alpha-7.csv};
\addplot3[proj] table[x=x1, y=y1, z expr=\zmin, col sep=comma] {data/inextensible-rod-OCP-result-alpha-7.csv};
 
\addplot3[very thick, color=TolVibrantOrange]
    table[x=x1, y=y1, z=z1, col sep=comma] {data/inextensible-rod-OCP-result-alpha-7.csv};
 
\addplot3[only marks, mark=*, mark size=1.5pt, red]
    coordinates {(0.4,0.4,0.12)};
\addplot3[only marks, mark=*, mark size=1.2pt, red, opacity=0.2]
    coordinates {(\xmax,0.4,0.12) (0.4,\ymax,0.12) (0.4,0.4,\zmin)};

\addplot3[only marks, mark=*, mark size=1.2pt, black]
     coordinates {(0.17181712472104044,0.5748729213074786,0.32009219983223997)};

\end{axis}
\end{tikzpicture}
\end{minipage}
\begin{minipage}[t]{0.32\textwidth}
    \begin{tikzpicture}
\begin{axis}[
    view={-20}{32},
    width=5.3cm, height=5cm,
    plot box ratio=1.3 1 1,
    xmin=\xmin, xmax=\xmax,
    ymin=\ymin, ymax=\ymax,
    zmin=\zmin, zmax=\zmax,
    xtick={0.2,0.3,0.4,0.5}, ytick={0.3,0.4,0.5,0.6}, ztick={0.1,0.2,0.3},
    scaled ticks=false,
    tick label style={font=\footnotesize},
    label style={font=\small},
    xlabel={$y_1$}, ylabel={$y_2$}, zlabel={$y_3$},
    grid=major, grid style={gray!25, line width=0.3pt},
]
 
\pgfplotsset{proj/.style={thin, opacity=0.5, TolVibrantOrange}}
\addplot3[proj] table[x expr=\xmax, y=y1, z=z1, col sep=comma] {data/inextensible-rod-OCP-result.csv};
\addplot3[proj] table[x=x1, y expr=\ymax, z=z1, col sep=comma] {data/inextensible-rod-OCP-result.csv};
\addplot3[proj] table[x=x1, y=y1, z expr=\zmin, col sep=comma] {data/inextensible-rod-OCP-result.csv};
 
\addplot3[very thick, color=TolVibrantOrange]
    table[x=x1, y=y1, z=z1, col sep=comma] {data/inextensible-rod-OCP-result.csv};
 
\addplot3[only marks, mark=*, mark size=1.5pt, red]
    coordinates {(0.4,0.4,0.12)};
\addplot3[only marks, mark=*, mark size=1.2pt, red, opacity=0.2]
    coordinates {(\xmax,0.4,0.12) (0.4,\ymax,0.12) (0.4,0.4,\zmin)};

\addplot3[only marks, mark=*, mark size=1.2pt, black]
     coordinates {(0.17181712472104044,0.5748729213074786,0.32009219983223997)};

\end{axis}
\end{tikzpicture}

\end{minipage}
\caption{Inextensible rod solution (orange) for $\alpha = 10^{-5}$ (left), $\alpha= 10^{-7}$ (middle), $\alpha=10^{-9}$ (right), desired state $y_d = (0.4,0.4,0.12)^T$ (red) and one-sided Dirichlet boundary conditions (black). The lines and the dot in light colors depict projections onto the coordinate planes.}
\label{fig:rod}
\end{figure}

For the required vector transport we use again an approach that relies on pointwise application of a projection according to \eqref{eq:VTH1}. It is well known that the projection onto a tangent space $T_X St(n,k)$ is given by the formula
\[
 P_X A = A-\frac12 X(X^TA+A^TX).
\]
Then, as a local vector transport for $\delta v \in T_{\hat v} H^1(\Omega,St(3,d))$ to $T_{v}H^1(\Omega,St(3,d))$ we choose
\[
 (\VTforward{\hat v}(v)\delta v)(\omega) := P_{v(\omega)}\delta v(\omega).
\]

Figure~\ref{fig:rod} shows numerical examples for the case $d=1$, i.e. an inextensible, controlled rod. The aim is to minimize the integral over the squared distance between the points of the rod from $y_d \equiv (0.4,0.4,0.12)^T$. Observe, how the rod winds itself around $y_d$, if the control cost is small. 
\begin{figure}[h]

    \centering
    \includegraphics[scale=0.75]{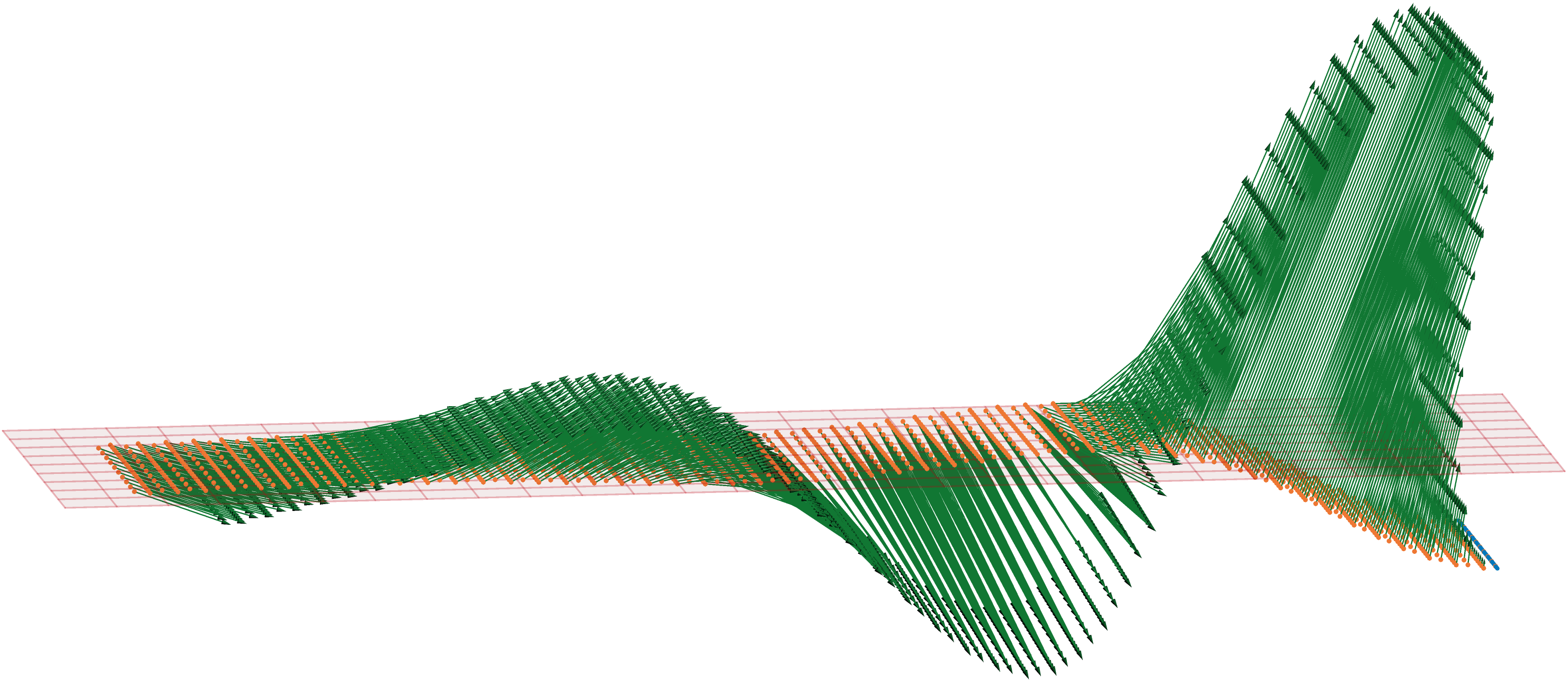}
\caption{Resulting sheet (orange) with Dirichlet boundary conditions (blue). The desired state $y_d$ (red) is the horizontal plane $w_3 = 0.4$. The control is illustrated by green arrows.}
\label{fig:shell}
\end{figure}
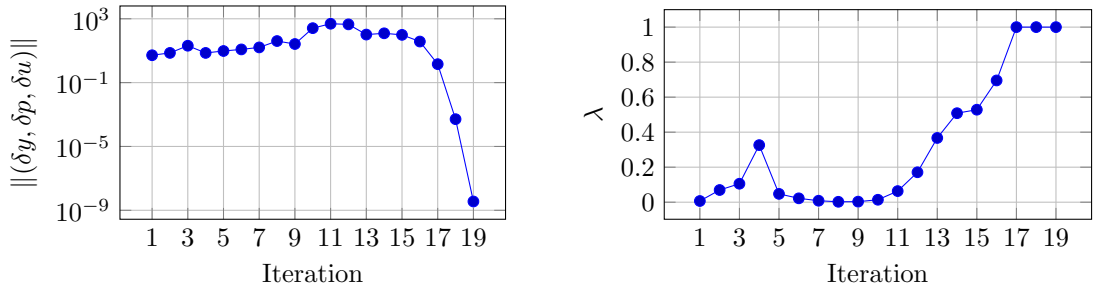
\begin{figure}[h]
\begin{minipage}[t]{0.49\textwidth}
        \begin{tikzpicture}
        \begin{axis}[
            xlabel={Iteration},
            ylabel={$\Vert (\delta y, \delta p, \delta u)\Vert$},
            ymode=log,
            grid=both,
            width=15cm,
            height=9cm,
            scale=0.38,
            xtick={1,3,5,7,9,11,13,15,17,19},
            ytick distance=10^4,
            xtick distance=2
        ]
                \addplot table [
            col sep=comma,
            header=false,
            x expr=\coordindex+1,   
            y index=0             
            ]{data/norm_newton_direction_sheet_OCP.csv};
        \end{axis}
        \end{tikzpicture}
\end{minipage}
    \begin{minipage}[t]{0.49\textwidth}
                \begin{tikzpicture}
        \begin{axis}[
            xlabel={Iteration},
            ylabel={$\lambda$},
            grid=both,
            width=11cm,
            height=6.2cm,
            scale=0.6,
            xtick={1,3,5,7,9,11,13,15,17,19},
            ytick distance=0.2,
            xtick distance=2
        ]
                \addplot table [
            col sep=comma,
            header=false,
            x expr=\coordindex+1,   
            y index=0             
            ]{data/stepsizes_sheet_OCP.csv};
        \end{axis}
        \end{tikzpicture}
        \label{fig:stepsizesRod}
    \end{minipage}
\caption{Superlinear convergence of Algorithm~\ref{alg:ImplementableNewton} (left) and computed stepsizes $\lambda\in(0,1]$ (right).}
\label{fig:shellConvergence}
\end{figure}

For the computation of a solution, we use linear Lagrange elements for  $w,v,u,p_w$, and $p_v$, and piecewise constant elements for $\lambda$ and $p_\lambda$ on a uniform partition of the unit interval. 
Since the problem is highly nonlinear, in particular for small $\alpha$, we used a simple homotopy method, decreasing the regularization parameter $\alpha$ in each iteration. Starting from $\alpha = 0.1$ we reduced $\alpha$ sucessively by a factor $(0.1)^{1/3}$ in every homotopy step.

For the case $d=2$, the controlled inextensible shell, we choose a strip of length $5$ and width $1$, which is clamped horizontally at one of the short edges. The aim is to bend the strip upwards so that its integrated square distance to the horizontal plane at height $w_3=0.4$ is minimized, balanced with the control cost. Figure~\ref{fig:shell} shows the computed result for $\alpha=10^{-3}$, including the bending forces. 

For its discretization we used finite element spaces, similar to those proposed in \cite{Bartels2013}: for $w$ and $p_w$ Crouzeix-Raviart-finite elements were used, for $v$ and $p_v$, linear finite elements, augmented by a bubble function were employed. For $\lambda$ and $p_\lambda$ we used again piecewise constant finite elements. 
The problem was solved by a damped Newton method, where an affine covariant damping strategy (cf. \cite[Alg. 2]{weigl2025newton}) was applied for globalization. Figure~\ref{fig:shellConvergence} shows the convergence history of Algorithm~\ref{alg:ImplementableNewton} and the computed damping factors. We observe an initial phase, where the globalization is active, and finally, local superlinear convergence.

\section{Conclusion}

This work can be seen as a first step into the algorithmic treatment of problems in geometric PDE-constrained optimization. The construction of a Newton method for solving the optimality system was enabled by a local splitting, based on vector transports. The computation of the second covariant derivative could be performed by a moving test function technique. Despite the high complexity of the problem, all quantities can be derived from relatively simple algorithmic ingredients, which are available for most of the standard manifolds and admit an efficient implementation.

In this paper we concentrated on the detailed description of the overall framework and the derivation of Newton's method. We plan to discuss aspects, concerning the discretization with finite elements, including details of the necessary constructions in a companion paper. 

Certainly, a number of important theoretical and algorithmic issues are still open and subject to future research. First of all, a functional analytic treatment of concrete classes of problems needs to be done. In particular, our abstract assumptions, e.g. on differentiability, need to be verified for appropriate classes of problems. An important related question is the integration of inequality constraints into our framework. Finally, algorithms that are more specifically constructed to constrained optimization on manifolds, refining, e.g., the ideas of \cite{SchielaOrtiz:2021}, need to be developed.  
\bibliographystyle{plain}
\bibliography{references}
\end{document}

%% file: macros.tex
\newcommand{\E}{\mathcal{E}}

\newcommand{\M}{\mathcal{M}}
\newcommand{\X}{\mathcal{X}}

\newcommand{\Y}{\mathcal{Y}}

\newcommand{\R}{\mathbb{R}}
\renewcommand{\L}{\mathcal{L}}

\newcommand{\DualBundleProjection}{\pi^*}
\newcommand{\coloneqq}{:=}

\newcommand{\id}{\mathrm{id}}

\newcommand{\NE}[1]{{\big\langle{#1}\big\rangle}}

\makeatletter
\newcommand{\oset}[3][0ex]{%
  {\mathop{#3}\limits^{
    \vbox to#1{\kern-2\ex@
    \hbox{\!$#2$}}}}}
\makeatother

\newcommand{\VT}[2]{\oset[-0.35ex]{\scriptscriptstyle#2}{V}_{\!\!#1}}
\newcommand{\VTback}[1]{\oset[-0.35ex]{\scriptscriptstyle\leftarrow}{V}_{\!\!#1}}
\newcommand{\VTforward}[1]{\oset[-0.35ex]{\scriptscriptstyle\rightarrow}{V}_{\!\!#1}}
\newcommand{\VVTforward}[1]{\oset[0.0ex]{\scriptscriptstyle\rightrightarrows}{V}_{\!\!#1}}
\newcommand{\WWTforward}[1]{\oset[0.0ex]{\scriptscriptstyle\rightrightarrows}{W}_{\!\!#1}}

\newcommand{\WTforward}[1]{\oset[-0.35ex]{\scriptscriptstyle\rightarrow}{W}_{\!\!#1}}
\newcommand{\VTprime}[2]{{\oset[-0.6ex]{\scriptscriptstyle#2}{V'}_{\!\!\!\!#1}}}
\newcommand{\WTprime}[2]{{\oset[-0.6ex]{\scriptscriptstyle#2}{W'}_{\!\!\!\!#1}}}
\newcommand{\VTstar}[1]{\oset[-0.4ex]{\!\!\scriptscriptstyle\leftarrow}{V^{*}}_{\!\!\!\!\!\!#1}\;}

%% file: data/Vektortransport.pdf_tex
\begingroup%
  \makeatletter%
  \providecommand\color[2][]{%
    \errmessage{(Inkscape) Color is used for the text in Inkscape, but the package 'color.sty' is not loaded}%
    \renewcommand\color[2][]{}%
  }%
  \providecommand\transparent[1]{%
    \errmessage{(Inkscape) Transparency is used (non-zero) for the text in Inkscape, but the package 'transparent.sty' is not loaded}%
    \renewcommand\transparent[1]{}%
  }%
  \providecommand\rotatebox[2]{#2}%
  \newcommand*\fsize{\dimexpr\f@size pt\relax}%
  \newcommand*\lineheight[1]{\fontsize{\fsize}{#1\fsize}\selectfont}%
  \ifx\svgwidth\undefined%
    \setlength{\unitlength}{370.37416473bp}%
    \ifx\svgscale\undefined%
      \relax%
    \else%
      \setlength{\unitlength}{\unitlength * \real{\svgscale}}%
    \fi%
  \else%
    \setlength{\unitlength}{\svgwidth}%
  \fi%
  \global\let\svgwidth\undefined%
  \global\let\svgscale\undefined%
  \makeatother%
  \begin{picture}(0.9,0.302657469)%
    \lineheight{1}%
    \setlength\tabcolsep{0pt}%
    \put(0.3,0){\includegraphics[scale=0.4,page=1]{data/Vektortransport.pdf}}%
    \put(0.40593244,0.1251338703){\color[rgb]{0,0,0}\makebox(0,0)[lt]{\lineheight{1.25}\smash{\begin{tabular}[t]{l}$\hat{y}$\end{tabular}}}}%
    \put(0.5333688,0.13268998){\color[rgb]{0,0,0}\makebox(0,0)[lt]{\lineheight{1.25}\smash{\begin{tabular}[t]{l}$y$\end{tabular}}}}%
    \put(0.39348351,0.17786222){\color[rgb]{0,0,0}\makebox(0,0)[lt]{\lineheight{1.25}\smash{\begin{tabular}[t]{l}$\hat{e}$\end{tabular}}}}%
    \put(0.525394583,0.17990135){\color[rgb]{0,0,0}\makebox(0,0)[lt]{\lineheight{1.25}\smash{\begin{tabular}[t]{l}$e$\end{tabular}}}}%
    \put(0.42853552,0.28183978){\color[rgb]{0,0,0}\makebox(0,0)[lt]{\lineheight{1.25}\smash{\begin{tabular}[t]{l}$\VT{\hat{y}}{\rightarrow}(y)$\end{tabular}}}}%
    \put(0.43369228,0.02810717){\color[rgb]{0,0,0}\makebox(0,0)[lt]{\lineheight{1.25}\smash{\begin{tabular}[t]{l}$E_{\hat{y}}$\end{tabular}}}}%
    \put(0.54937263,0.02446334){\color[rgb]{0,0,0}\makebox(0,0)[lt]{\lineheight{1.25}\smash{\begin{tabular}[t]{l}$E_{y}$\end{tabular}}}}%
    \put(0.69333688,0.10268998){\color[rgb]{0,0,0}\makebox(0,0)[lt]{\lineheight{1.25}\smash{\begin{tabular}[t]{l}$\Y$\end{tabular}}}}%
    \put(0.3,0){\includegraphics[scale=0.4,page=2]{data/Vektortransport.pdf}}%
  \end{picture}%
\endgroup%

%% file: references.bib
@article {Bartels2013,
    AUTHOR = {Bartels, S\"oren},
     TITLE = {Finite element approximation of large bending isometries},
   JOURNAL = {Numer. Math.},
  FJOURNAL = {Numerische Mathematik},
    VOLUME = {124},
      YEAR = {2013},
    NUMBER = {3},
     PAGES = {415--440},
      ISSN = {0029-599X,0945-3245},
   MRCLASS = {65N30 (74K20 74S05)},
  MRNUMBER = {3066035},
MRREVIEWER = {Igor\ Bock},
       DOI = {10.1007/s00211-013-0519-7},
       URL = {https://doi.org/10.1007/s00211-013-0519-7},
}

@article{Bergmann2022,
    Author    = {Ronny Bergmann},
    Doi       = {10.21105/joss.03866},
    Journal   = {Journal of Open Source Software},
    Number    = {70},
    Pages     = {3866},
    Publisher = {The Open Journal},
    Title     = {{M}anopt.jl: {O}ptimization on {M}anifolds in {J}ulia},
    Volume    = {7},
    Year      = {2022},
}

@article{weigl2024newton,
  title={Newton's method for nonlinear mappings into vector bundles},
  author={Weigl, Laura and Schiela, Anton},
  journal={arXiv preprint arXiv:2404.04073},
  year={2024}
}

@article{weigl2025newton,
  title={Newton's method for nonlinear mappings into vector bundles {P}art {II}: {A}pplication to variational problems},
  author={Weigl, Laura and Bergmann, Ronny and Schiela, Anton},
  journal={arXiv preprint arXiv:2507.13836},
  year={2025}
}

@software{carlsson_ferrite_2025,
  author    = {Carlsson, Kristoffer and Ekre, Fredrik and {Ferrite.jl contributors}},
  title     = {{Ferrite.jl}},
  version   = {1.2.0},
  year      = {2025},
  publisher = {Zenodo},
  doi       = {10.5281/zenodo.17608426},
  note       = {https://github.com/Ferrite-FEM/Ferrite.jl}
}

@article{loayza2026retraction,
  title={Retraction based regression methods on Riemannian manifolds},
  author={Loayza-Romero, Estefan{\'\i}a and Sibum, Benedikt and Welker, Kathrin},
  journal={arXiv preprint arXiv:2606.02116},
  year={2026}
}

@book{troltzsch2010optimal,
  title={Optimal control of partial differential equations: theory, methods, and applications},
  author={Tr{\"o}ltzsch, Fredi},
  volume={112},
  year={2010},
  publisher={American Mathematical Soc.}
}

@book{hinze2008optimization,
  title={Optimization with PDE constraints},
  author={Hinze, Michael and Pinnau, Ren{\'e} and Ulbrich, Michael and Ulbrich, Stefan},
  year={2008},
  publisher={Springer Science \& Business Media}
}

@article{desimone2002reduced,
  title={A reduced theory for thin-film micromagnetics},
  author={DeSimone, Antonio and Kohn, Robert V and M{\"u}ller, Stefan and Otto, Felix},
  journal={Communications on Pure and Applied Mathematics: A Journal Issued by the Courant Institute of Mathematical Sciences},
  volume={55},
  number={11},
  pages={1408--1460},
  year={2002},
  publisher={Wiley Online Library}
}

@inproceedings{heeren2012time,
  title={Time-discrete geodesics in the space of shells},
  author={Heeren, Behrend and Rumpf, Martin and Wardetzky, Max and Wirth, Benedikt},
  booktitle={Computer Graphics Forum},
  volume={31},
  number={5},
  pages={1755--1764},
  year={2012},
  organization={Wiley Online Library}
}

@article{weinmann2014total,
  title={Total variation regularization for manifold-valued data},
  author={Weinmann, Andreas and Demaret, Laurent and Storath, Martin},
  journal={SIAM Journal on Imaging Sciences},
  volume={7},
  number={4},
  pages={2226--2257},
  year={2014},
  publisher={SIAM}
}

@article{hardt1986existence,
  title={Existence and partial regularity of static liquid crystal configurations},
  author={Hardt, Robert and Kinderlehrer, David and Lin, Fang-Hua},
  journal={Communications in mathematical physics},
  volume={105},
  number={4},
  pages={547--570},
  year={1986},
  publisher={Springer}
}

@book{agrachev2013control,
  title={Control theory from the geometric viewpoint},
  author={Agrachev, Andrei A and Sachkov, Yuri},
  year={2013},
  publisher={Springer Science \& Business Media}
}

@book{schattler2012geometric,
  title={Geometric optimal control: theory, methods and examples},
  author={Sch{\"a}ttler, Heinz and Ledzewicz, Urszula},
  volume={38},
  year={2012},
  publisher={Springer}
}

@article{goodman2026riemannian,
  title={Riemannian Metric Preconditioning for Trajectory Tracking},
  author={Goodman, Jacob R and Jasa, Hajg},
  journal={arXiv preprint arXiv:2606.18913},
  year={2026}
}

@book{bartels2015numerical,
  title={Numerical methods for nonlinear partial differential equations},
  author={Bartels, S{\"o}ren},
  volume={47},
  year={2015},
  publisher={Springer}
}

@incollection{hardering2020geometric,
  title={Geometric finite elements},
  author={Hardering, Hanne and Sander, Oliver},
  booktitle={Handbook of variational methods for nonlinear geometric data},
  pages={3--49},
  year={2020},
  publisher={Springer}
}

@incollection{hajlasz2009sobolev,
  title={Sobolev mappings between manifolds and metric spaces},
  author={Hajlasz, Piotr},
  booktitle={Sobolev Spaces In Mathematics I: Sobolev Type Inequalities},
  pages={185--222},
  year={2009},
  publisher={Springer}
}

@article{Julia,
    title={{J}ulia: {A} fresh approach to numerical computing},
    author={Bezanson, Jeff and Edelman, Alan and Karpinski, Stefan and Shah, Viral B},
    journal={SIAM {R}eview},
    volume={59},
    number={1},
    pages={65--98},
    year={2017},
    publisher={SIAM},
    doi={10.1137/141000671},
    url={https://epubs.siam.org/doi/10.1137/141000671}
}

@book{Lang,
author = {Lang, Serge},
title = {{F}undamentals of differential geometry},
publisher ={Springer Science \& Business Media},
volume = {191},
year = {2012}
}

@article{BezansonEdelanKarpinskiViral:2017,
    AUTHOR    = {Bezanson, Jeff and Edelman, Alan and Karpinski, Stefan and Shah, Viral B},
    JOURNAL   = {SIAM review},
    NUMBER    = {1},
    PAGES     = {65--98},
    PUBLISHER = {SIAM},
    TITLE     = {Julia: {A} fresh approach to numerical computing},
    DOI       = {10.1137/141000671},
    VOLUME    = {59},
    YEAR      = {2017}
}

@book{absil2008optimization,
  title={{O}ptimization algorithms on matrix manifolds},
  author={Absil, P-A and Mahony, Robert and Sepulchre, Rodolphe},
  year={2008},
  publisher={Princeton University Press}
}

@book{boumal2023intromanifolds,
  author    = {Boumal, Nicolas},
  title     = {{A}n introduction to optimization on smooth manifolds},
  publisher = {Cambridge University Press},
  year      = {2023},
  url       = {https://www.nicolasboumal.net/book},
  doi       = {10.1017/9781009166164}
}

@article{SchielaOrtiz:2021,
author = {Schiela, Anton and Ortiz, Julian},
title = {An SQP Method for Equality Constrained Optimization on Hilbert Manifolds},
journal = {SIAM Journal on Optimization},
volume = {31},
number = {3},
pages = {2255-2284},
year = {2021},
doi = {10.1137/20M1341325},
}

@online{AxenBaranBergmannRzecki:2023, 
    author = {Seth D. Axen and Mateusz Baran and Ronny Bergmann and Krzysztof Rzecki},
    title = {{M}anifolds.jl: {A}n {E}xtensible {J}ulia {F}ramework for {D}ata {A}nalysis on {M}anifolds},
    year = {2021},
    eprint = {2106.08777},
    eprinttype = {arXiv},
}

@book {MarsdenRaitu:1994,
    AUTHOR = {Marsden, Jerrold E. and Ratiu, Tudor S.},
     TITLE = {Introduction to mechanics and symmetry},
    SERIES = {Texts in Applied Mathematics},
    VOLUME = {17},
      NOTE = {A basic exposition of classical mechanical systems},
 PUBLISHER = {Springer-Verlag, New York},
      YEAR = {1994},
     PAGES = {xvi+500},
      ISBN = {0-387-97275-7; 0-387-94347-1},
   MRCLASS = {58F05 (58-01 70-01 70Hxx)},
  MRNUMBER = {1304682},
MRREVIEWER = {Gerard\ Thompson},
       DOI = {10.1007/978-1-4612-2682-6},
       URL = {https://doi.org/10.1007/978-1-4612-2682-6},
}

@Inbook{Schulz2015,
author="Schulz, Volker H.
and Siebenborn, Martin
and Welker, Kathrin",
editor="Pratelli, Aldo
and Leugering, G{\"u}nter",
title="Towards a Lagrange--Newton Approach for PDE Constrained Shape Optimization",
bookTitle="New Trends in Shape Optimization",
year="2015",
publisher="Springer International Publishing",
address="Cham",
pages="229--249",
isbn="978-3-319-17563-8",
doi="10.1007/978-3-319-17563-8_10",
url="https://doi.org/10.1007/978-3-319-17563-8_10"
}

@article{bergmann2022first,
  title={First-and second-order analysis for optimization problems with manifold-valued constraints},
  author={Bergmann, Ronny and Herzog, Roland and Ortiz L{\'o}pez, Juli{\'a}n and Schiela, Anton},
  journal={Journal of Optimization Theory and Applications},
  volume={195},
  number={2},
  pages={596--623},
  year={2022},
  publisher={Springer}
}

@article{zeidler1986nonlinear,
  title={Nonlinear {F}unctional {A}nalysis and {I}t's {A}pplications},
  author={Zeidler, Eberhard},
  journal={Fixed-Point Theorems},
  pages={897},
  year={1986},
  publisher={Springer-Verlag}
}

@article{helein2008harmonic,
  title={Harmonic maps},
  author={H{\'e}lein, Fr{\'e}d{\'e}ric and Wood, John C},
  journal={Handbook of global analysis},
  volume={1213},
  pages={417--491},
  year={2008},
  publisher={Elsevier Amsterdam}
}
